\ExplSyntaxOn
\str_clear:N \c_sys_backend_str
\ExplSyntaxOff
\documentclass[pdflatex,sn-mathphys-ay]{sn-jnl}

\usepackage{tikz}%
\usetikzlibrary{calc}%
\usepackage{graphicx}%
\usepackage{amsmath,amssymb,amsfonts}%
\usepackage{amsthm}%
\usepackage{mathrsfs}%
\usepackage[title]{appendix}%
\usepackage{textcomp}%
\usepackage{manyfoot}%
\usepackage{booktabs}%
\usepackage{listings}%
\usepackage{geometry}%
\usepackage{ulem}%
\usepackage{here}%
\usepackage{stmaryrd}%
\usepackage{arydshln}%
\usepackage{mathtools}%
\usepackage{bm}%
\usepackage{lscape}%
\usepackage{enumitem}%
\usepackage{natbib}%
\usepackage{contour}%
\contourlength{0.05em}

\theoremstyle{thmstyleone}
\newtheorem{thm}{Theorem}
\newtheorem{proposition}{Proposition}
\newtheorem{lema}{Lemma}
\newtheorem{corollary}{Corollary}
\newtheorem{conj}{Conjecture}

\theoremstyle{thmstyletwo}
\newtheorem{ex}{Example}
\newtheorem{ques}{Question}
\newtheorem{rem}{Remark}
\newtheorem*{ntt}{Notation}

\theoremstyle{thmstylethree}
\newtheorem{definition}{Definition}
\newtheorem{defntt}{Definition-Notation}
\newtheorem{conv}{Convention}

\DeclareMathOperator*{\mex}{mex}

\newcommand{\exw}{1.2}
\newcommand{\egw}{0.52}
\newcommand{\epw}{1.6}

\newcommand{\repw}{0.7}

\newcommand{\haa}{6.8}
\newcommand{\hcc}{19/8*1.5+\repw+4/8*\haa}

\newcommand{\fgx}{11.6}

\newcommand{\aaa}{4.6}
\newcommand{\bbb}{5.98}
\newcommand{\bbc}{6.44}
\newcommand{\bbd}{7.36}
\newcommand{\mmm}{\aaa+\bbb}
\newcommand{\mmc}{\aaa+\bbc}
\newcommand{\mmd}{\aaa+\bbd}

\newcommand{\ext}{2.8}
\newcommand{\exa}{\ext*sqrt(11)}
\newcommand{\exb}{\ext*sqrt(23)}

\newcommand{\cod}[2]{\draw[fill=black] ({#1},2) circle[radius=.04]
node [above] {$#2$};}

\newcommand{\codd}[2]{\draw[fill=black] ({#1},0) circle[radius=.04]
node [above] {$#2$};}
\newcommand{\coddd}[2]{\draw[fill=black] ({#1},-2) circle[radius=.04]
node [above] {$#2$};}

\newcommand{\gbox}[4]
{\draw ({#2},{2-.5*#1}) -- ({#2},{2.5-.5*#1}) [rounded corners=6pt] -- ({#3},{2.5-.5*#1}) -- ({#3},{2-.5*#1}) -- ({#2},{2-.5*#1});
\node at ({#2},{2.25-.5*#1}) [right] {$#4$};}

\newcommand{\gboxs}[4]
{\draw ({#2},{2-.5*#1}) -- ({#2},{2.5-.5*#1}) [rounded corners=6pt] -- ({#3},{2.5-.5*#1}) -- ({#3},{2-.5*#1}) -- ({#2},{2-.5*#1});
\node at ({#2-.08},{2.25-.5*#1}) [right] {$#4$};}
\newcommand{\gboxx}[4]
{\draw ({#3},{2-.5*#1}) -- ({#3},{2.5-.5*#1}) [rounded corners=6pt] -- ({#2},{2.5-.5*#1}) -- ({#2},{2-.5*#1}) -- ({#3},{2-.5*#1});
\node at ({#2+.2},{2.25-.5*#1}) [right] {$#4$};}
\newcommand{\gboxxx}[4]
{\draw ({#2},{2-.5*#1}) -- ({#2},{2.5-.5*#1}) -- ({#3},{2.5-.5*#1}) -- ({#3},{2-.5*#1}) -- ({#2},{2-.5*#1});
\node at ({#2+.2},{2.25-.5*#1}) [right] {$#4$};}

\newcommand{\lper}[2]
{\draw[line width=.08cm] ({#2+.04},{2-.5*#1}) -- ({#2+.04},{2.5-.5*#1});}
\newcommand{\lperr}[2]
{\draw[line width=.08cm] ({#2+.16},{2.46-.5*#1}) [rounded corners=6pt] -- 
({#2+.02},{2.46-.5*#1}) -- ({#2+.02},{2.04-.5*#1}) -- ({#2+.16},{2.04-.5*#1});}

\newcommand{\rper}[2]
{\draw[line width=.08cm] ({#2-.16},{2.46-.5*#1}) [rounded corners=6pt] -- 
({#2-.02},{2.46-.5*#1}) -- ({#2-.02},{2.04-.5*#1}) -- ({#2-.16},{2.04-.5*#1});}
\newcommand{\rperr}[2]
{\draw[line width=.08cm] ({#2-.04},{2-.5*#1}) -- ({#2-.04},{2.5-.5*#1});}

\renewcommand{\theenumi}{\arabic{enumi}}

\makeatletter
\renewcommand{\p@enumii}{}
\makeatother

\newtagform{bbracket}{$\langle$}{$\rangle$}
\usetagform{bbracket}

\begin{document}

\title[Article Title]{Subtraction Nim with Continuous Parameters}

\author*[1]{\fnm{Yuto} \sur{Moriwaki}}\email{d252214@hiroshima-u.ac.jp}

\affil*[1]{\orgdiv{Department of Mathematics}, \orgname{Hiroshima University}, 
\orgaddress{
%\street{xx}, 
\city{Hiroshima}, 
%\postcode{xx}, \state{xx}, 
\country{Japan}}}

\abstract{When $S$ is a finite set of positive integers, we can consider classical Subtraction Nim with $S$ as the set of removable numbers. Even when $S$ consists of three elements, many questions remain unanswered. For example, we do not have a period formula of the Nim value.
In this paper, we generalize $S$ to be a finite set of positive real numbers. We found that in some regions, we can give concrete formulae for the period and the Nim value function. In particular when $S$ consists of three elements, we found sufficient conditions for the Nim value function to be purely periodic with the period which is equal to the sum of two of elements of $S$. To be more precise, let $S=\{a,b,c\}$ with $0<a<b<c$, then for example when $a\leq b\leq 2a$ with $a+b\geq c$, the Nim value function is purely periodic with a period $a+c$. There are much more regions with precise period formulae. We have also some generalizations for the cases $|S|\geq4$. Even when $S$ consists of integers, these results seem to be new.}


\maketitle
\bmhead{Acknowledgements}

This work is supported by JST SPRING, Grant Number JPMJSP2132.

\section{Introduction}

The classical Subtraction Nim is a heap game with a single heap, where a finite set $S$ of positive integers is given and the player takes $s\in S$ tokens from the heap. We consider the normal play, where two players play alternately and the one who cannot play loses. There is an algorithm to calculate the Nim value $g_S(n)$ with given $S$ and $n$, the size of the heap. It is known that the sequence of Nim values $g_S(0), g_S(1), g_S(2),\ldots$ is ultimately periodic, that is to say, there exists $n_0$ and $p>0$ such that for all $n\geq n_0$, the Nim values satisfy $g_S(n+p)=g_S(n)$. However, the formulae to calculate such $p$ and $n_0$ from $S$ directly are yet unknown, even with the condition $|S|=3$. In particular, the condition for $n_0=0$ (namely the condition for pure periodicity) seem to be hardly founded. See \citep{flammenkamp1997lange} for some results when $|S|\geq4$.

In this paper, we consider the Subtraction Nim with $S$ consisting of any positive real numbers, not necessarily integers. We especially examine the case $S=\{a,b,c\}$ (three elements) with the conditions $0<a<b<c$ and additionally $c=1$, as we can rescale. This generalization enable us to take Subtraction Nim not as a heap game, but as a ``string game,'' where a single string is given and the player trims length $s\in S$ off the string from an end. We do have a complicated example for general $S$ (Example~\ref{ex:1r2r6}), but on the other hand, we observed the phenomenon that for $(a,b)$ in some particular regions, the Nim value function turns to be purely periodic and its period is also accurately calculatable (Theorem~\ref{th:necktie}).

The idea for the proof is to list all the possible points where the Nim value may change, calling each of them a break, and get the Nim value between adjacent breaks in ascending order. It is just the same as the ordinary algorithm to calculate the Nim value, but we proceed the calculation with parameters, so each result is for all of $S$ corresponding to each region. \citep{ward2016conjecture} conjectured that for the Subtraction Nim with $S=\{a,b,c\}$ which satisfies $0<a<b<c$ and $c\neq a+b$ the Nim value function has a period equal to either of $a+b$, $b+c$ or $c+a$, and our result supports it for infinitely many cases.

\begin{ntt}
\ 
\begin{enumerate}
\item When $S$ is a subset of $\mathbb{R}_{>0}$, then we let $S_{\leq x}\coloneqq \{s\in S\mid s\leq x\}$.
\item When $x\in\mathbb{R}$, then $\lfloor x\rfloor $ is the largest integer $ \leq x$ and $\lceil x\rceil $ is the least integer $ \geq x$.
\item The notation $ [ x_n]_n$ means a sequence of real numbers $x_0,x_1,x_2,\ldots$ with the index $ n\in \mathbb{Z}_{\geq 0}$.
\item When $ b\in \mathbb{R}$ and $ a\in \mathbb{R} _{>0}$, then $b\bmod a$ is the remainder of $b$ divided by $a$, namely
\begin{equation*}
 b\bmod a=b-\left\lfloor\frac{b}{a}\right\rfloor a.
\end{equation*}

\end{enumerate}

\end{ntt}

\section{Preparation}

\begin{definition}

Let $S$ be a nonempty finite subset of $\mathbb{R}_{>0}$. We call the Subtraction Nim with the set of removable numbers $S$ as \underline{Continim} $G(S)$, a shorthand for ``continuous nim''. We consider the normal play in this paper. $g_S:\mathbb{R}_{\geq 0}\rightarrow \mathbb{Z}_{\geq 0}$, or simply $ g$, denotes the Nim value function of Continim $ G(S)$, and we call each value $g_S(x)$ as the Nim value of the position $x\in\mathbb{R}_{\geq 0}$. 

\end{definition}

\begin{proposition}\label{prop:fg}

Let $S$ be a nonempty finite subset of $\mathbb{R}_{>0}$. If a function $ f:\mathbb{R}_{\geq 0}\rightarrow \mathbb{Z}_{\geq 0}$ satisfies
\begin{equation*}
 f( x) =\mex\{f( x-s) \mid s\in S_{\leq x}\} \ ( x\in \mathbb{R}_{\geq 0}),
\end{equation*}
then $ f$ is identical to $ g_S$, the Nim value function of Continim $ G(S)$.

\end{proposition}

\begin{proof}

Let $S=\{s_1,s_2,\ldots,s_r\}$ with $0<s_1<s_2<\cdots<s_r$ and we proceed by induction on $k\coloneqq\displaystyle\left\lfloor\frac{x}{s_1}\right\rfloor$. When $k=0$, then $x<s_1$ is a terminal position, and
\begin{equation*}
f(x)=\mex\emptyset=0=g_S(x).
\end{equation*}
When $k>0$, by inductive assumption, we may assume that $f(y)=g_S(y)$ holds for $y=x-s$ with $s\in S_{\leq x}$. Then
\begin{align*}
f(x)&=\mex\{f(x-s)\mid s\in S_{\leq x}\}\\
    &=\mex\{g_S(x-s)\mid s\in S_{\leq x}\}\\
    &=g_S(x).
\end{align*}

\end{proof}

\begin{proposition}\label{prop:ord}

Let $S$ be a finite subset of $\mathbb{Z}_{>0}$, $g:\mathbb{R}_{\geq 0}\rightarrow \mathbb{Z}_{\geq 0}$ be the Nim value function of Continim, and $g_0:\mathbb{Z}_{\geq 0}\rightarrow \mathbb{Z}_{\geq 0}$ be that of the classical Subtraction Nim with the same set of removable numbers $S$. Then we have
\begin{equation*}
g(x)=g_0(\lfloor x\rfloor)\ (x\in\mathbb{R}_{\geq 0}).
\end{equation*}

\end{proposition}

\begin{proof}

By the definition of the Nim value,
\begin{align*}
g_0(\lfloor x\rfloor)
&=\mex\{g_0(\lfloor x\rfloor -s)\mid s\in S_{\leq \lfloor x\rfloor}\}\\
&=\mex\{g_0(\lfloor x-s \rfloor)\mid s\in S_{\leq x}\}\ (x\geq 0).
\end{align*}
Thus the function $f(x)\coloneqq g_0(\lfloor x\rfloor)$ satisfies
\begin{equation*}
f(x)=\mex\{f(x-s)\mid s\in S_{\leq x}\}\ (x\geq 0).
\end{equation*}
By Proposition~\ref{prop:fg}, we get $f=g$.

\end{proof}

We get $ g_{0} =g|_{\mathbb{Z}_{\geq 0}}$ by Proposition~\ref{prop:ord}, so we will write just $G(S)$ to consider Continim.

\begin{definition}\label{df:period}

A function $ f:\mathbb{R}_{\geq 0}\rightarrow \mathbb{Z}_{\geq 0}$ is said to be \underline{periodic} if there exist $x_0\in \mathbb{R}_{\geq 0}$ and $p\in \mathbb{R}$ such that the equality $f(x+p)=f(x)\ (x\geq x_0)$ holds. We call $ p$ (resp. $ x_{0}$) a \underline{period} (resp. \underline{preperiod}) of $ f$, and each 
\begin{equation*}
f|_{[x_0,x_0+p)},f|_{[x_0+p,x_0+2p)},f|_{[x_0+2p,x_0+3p)},\ldots
\end{equation*}
a \underline{loop} of $f$. The minimality is not required for them, and if $p$ is minimal among the periods of $f$, then we call $p$ the \underline{minimum period} of $f$. The same for the preperiod, and if the minimum preperiod is $0$, then we call the function $f$ to be \underline{purely periodic} and a period $p$ to be a \underline{pure period} of $f$. 

\end{definition}

For the classical Subtraction Nim, the Nim values can be calculated by induction, and the same for the Nim value of each position in Continim. But it cannot be the way to find the Nim value function of Continim. Alternately, we introduce ``breaks''.

\begin{definition}

For given Continim $ G( S)$, we call each value $ \displaystyle\sum\limits _{s\in S} a_{s} s$ with some $ a_{s} \in \mathbb{Z}_{\geq 0}$ to be a \underline{break}. We index the breaks as $ 0=x_{0} < x_{1} < x_{2} < \cdots $ in the usual order of $ \mathbb{R}_{\geq 0}$.

Each break $ x_{n}$ can also be gained by induction, in the sense of calculating the Nim value, as follows:

$ x_{0} \coloneqq 0$. When $ x_{0} ,x_{1} ,\ldots ,x_{n-1}$ gained, let
\begin{equation*}
C( n) \coloneqq \{s+x_{l} \mid s\in S;\ l=0,1,\ldots ,n-1\},
\end{equation*}
then $ x_{n}$ is the least element of $ C( n)$ such that $>x_{n-1}$, i.e., $ x_{n} =\min C( n)_{ >x_{n-1}}$. (We call $ C( n)_{ >x_{n-1}}$ as $ n$th candidate-set.)

\end{definition}

\begin{thm} \label{th:const}

For $ j=1,2,3,\ldots $, the Nim value function $g$ of $G(S)$ is constant in $ [ x_{j-1} ,x_{j})$.

\end{thm}

\begin{proof}

Let $S=\{s_1,s_2,\ldots,s_r\}$ with $0<s_1<s_2<\cdots<s_r$ and we proceed by induction on $j$. For $j=0$, we have $x_0=0$ and $x_1=s_1$. The position $ x<x_1$ is terminal, so $g(x)=0$: constant in $[x_0,x_1)$. When $j>0$, by inductive assumption, we may assume that $g$ is constant in $[x_{k-1},x_k)$ for $k\leq j-1$. For any $ x\in [ x_{j-1} ,x_j)$ and $ s_i\leq x$, we have 
\begin{equation*}
x-s_i\in[x_{j-1}-s_i,x_j-s_i).
\end{equation*}
Let $j_i$ be the index such that
\begin{equation*}
x_{j_i-1}\leq x_{j-1}-s_i<x_{j_i},
\end{equation*}
then necessarily
\begin{equation*}
j_i\leq j-1,
\end{equation*}
and
\begin{equation*}
C(j)\ni s_i+x_{j_i}>x_{j-1}.
\end{equation*}
By the inducive definition of $x_j$, we get
\begin{align*}
    x_{j}&{}\leq s_i+x_{j_i},\\
x_{j}-s_i&{}\leq x_{j_i}.
\end{align*}
Hence
\begin{equation*}
[x_{j-1}-s_i,x_j-s_i)\subset[x_{j_i-1},x_{j_i}),
\end{equation*}
and then each $g(x-s_i)$ is constant for $x\in[x_{j-1},x_j)$.
 
\end{proof}

\begin{rem}

By Theorem~\ref{th:const}, the sequence of breaks $ [x_n]_n$ together with the data $g(x_n)\ (n=0,1,2,\ldots)$ is enough to determine the Nim value function. Theorem~\ref{th:const} also implies that the Nim value function $g$ is an upper semi-continuous step function, i.e., for any $x\in \mathbb{R}_{\geq 0}$ there exists $\varepsilon>0$ such that the equality $g(y)=g(x)\ (y\in [x,x+\varepsilon))$ holds,
but it does not imply the discontinuousness at $x=x_1,x_2,\ldots$.

\end{rem}

\begin{definition}

We call a break $ y$ to be a \underline{true break} if $ g$ is not continuous at $ y$. We index the true breaks as $ 0=y_{0} < y_{1} < y_{2} < \cdots $ in the usual order of $ \mathbb{R}_{\geq 0}$.

\end{definition}

\begin{rem}

$ S$ is nonempty for $ G( S)$ and every play changes the Nim value of the position, so there exist infinitely many true breaks.

\end{rem}

\begin{ex}\label{ex:1r2r6}

Let $ S\coloneqq \{1,\sqrt{2} ,\sqrt{6}\}$. The Nim value function of $ G( S)$ is shown in Fig.~\ref{fig:1r2r6} as a step graph.
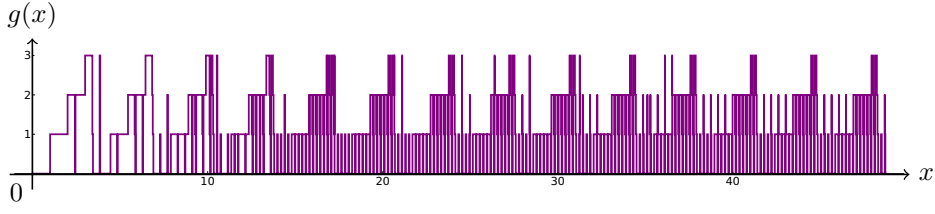
\begin{figure}[H]
\centering
\begin{tikzpicture} % Used: Mathematica and TikZ
\node at (5.65,.735) {\includegraphics[width=.9\linewidth]{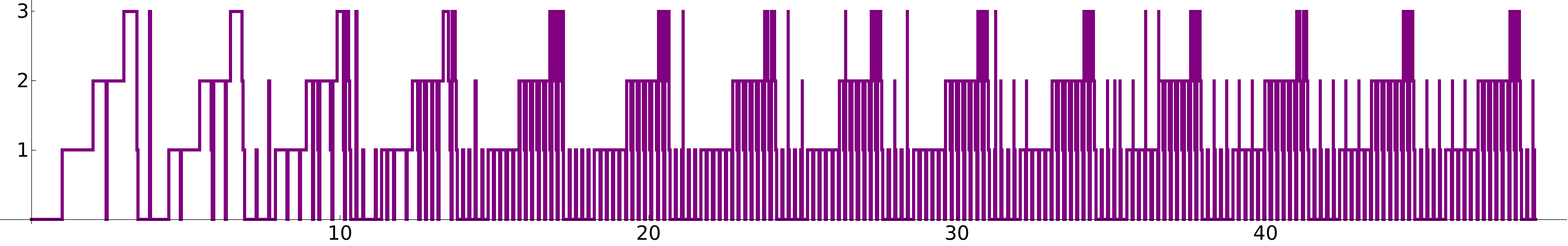}};
\draw [line width=0.6pt] (0,-.2) [->] --(0,1.8);
\draw [line width=0.6pt] (-.3,0) [->] --(11.6,0);
\node at (0,0) [below left] {$0$};
\node at (0,1.8) [above] {$g(x)$};
\node at (11.6,0) [right] {$x$};
\end{tikzpicture}
\caption{The Nim value function of $G(S)$ where $S=\{1,\sqrt{2},\sqrt{6}\}$.}\label{fig:1r2r6}
\end{figure}
Each vertical line indicates a true break. The minimum preperiod is equal to the 413th true break
\begin{equation*}
y_{413} =1+13\sqrt{2} +7\sqrt{6} \approx 36.531,
\end{equation*}
and the minimum period is $ 1+\sqrt{6}$.

\end{ex}

\begin{definition}

Let $ [ y_{n}]_{n}$ be the sequence of true breaks of $ G( S)$. For each $ n=0,1,2,\ldots $, take $ n'$ as the minimum such that $ y_{n'} -y_{n} \geq \max S$. Then we introduce an $ n$th \underline{period seed}
\begin{equation*}
\bm{Y}_{n} \coloneqq \begin{pmatrix}
y_{n+1} -y_{n} & y_{n+2} -y_{n+1} & \cdots  & y_{n'} -y_{n'-1}\\
g( y_{n}) & g( y_{n+1}) & \cdots  & g( y_{n'-1})
\end{pmatrix}.
\end{equation*}
For ease, let $ d_{i} \coloneqq y_{i+1} -y_{i}$. Then $\bm{Y}_n=\begin{pmatrix}
d_n & d_{n+1} & \cdots & d_{n'-1}\\
g(y_n) & g(y_{n+1}) & \cdots & g(y_{n'-1})
\end{pmatrix}$.

\end{definition}

\begin{proposition}

If there exist $ n_{1} ,n_{2} \in \mathbb{Z}_{\geq 0} \ ( n_{1} < n_{2})$ with $ \bm{Y}_{n_{1}} =\bm{Y}_{n_{2}}$, then the Nim value function of $ G(S)$ is periodic with a period $p=y_{n_2}-y_{n_1}$ and a preperiod $y_{n_1}$.

\end{proposition}

\begin{proof}

Let $S=\{s_1,s_2,\ldots,s_r\}$ with $0<s_1<s_2<\cdots<s_r$ and we prove
\begin{equation*}
g(x+p)=g(x)\ (x\geq y_{n_1})
\end{equation*}
by induction on $k=\left\lfloor \displaystyle\frac{x-y_{n_1}}{s_1}\right\rfloor $. Let $n'=(n_1)'$ for $n=n_1$. When $k=0$, we have 
\begin{align*}
y_{n_2+i}&=y_{n_2}+\sum_{j=n_2}^{n_2+i-1}d_j\\
         &=y_{n_1}+p+\sum_{j=n_1}^{n_1+i-1}d_j\\
         &=y_{n_1+i}+p\ (i=0,1,\ldots,(n_1)'-n_1-1),
\end{align*}
so if $y_{n_1+i}\leq x<y_{n_1+i+1}$, then by the assumption,
\begin{align*}
  y_{n_1+i}&\leq x<y_{n_1+i+1}\\
y_{n_1+i}+p&\leq x+p<y_{n_1+i+1}+p\\
  y_{n_2+i}&\leq x+p<y_{n_2+i+1}
\end{align*}
so
\begin{equation*}
 g(x)=g(y_{i+n_1})=g(y_{i+n_2})=g(x+p).
\end{equation*}
When $k>0$, by inductive assumption, we may assume that $g(y+p)=g(y)$ holds for $y=x-s$, $s\in S_{\leq x}$. Then
\begin{align*}
g(x+p)&=\mex\{g(x+p-s)\mid s\in S_{\leq x}\}\\
      &=\mex\{g(x-s)\mid s\in S_{\leq x}\}\\
      &=g_S(x).
\end{align*}

\end{proof}

\begin{ques}
Is the Nim value function of $ G(S)$ with a finite set $S$ always periodic?
\end{ques}

We have a conjecture for this question at the bottom of the next section. (Conjecture~\ref{conj:genper}).

\begin{rem}

Let $ n_{1}$ be the minimum such that $ \bm{Y}_{n_{1}} =\bm{Y}_{n_{2}}$ for some $ n_{2}  >n_{1}$ (assume it exists). This minimality does not imply that the minimum preperiod is $ y_{n_{1}}$ (unless $ n_{1} =0$).\\
Take $ n_{2}$ to be the minimum among $ \bm{Y}_{n_{2}} =\bm{Y}_{n_{1}}$ with $ n_{2}  >n_{1}  >0$. The minimality of $ n_{1} ,n_{2}$ implies $ \bm{Y}_{n_{1} -1} \neq \bm{Y}_{n_{2} -1}$, and then $ \bigl( d_{n_{1} -1} ,g( y_{n_{1} -1})\bigr) \neq \bigl( d_{n_{2} -1} ,g( y_{n_{2} -1})\bigr)$.\\
Consider the case $ g( y_{n_{1} -1}) =g( y_{n_{2} -1})$ and $ d_{n_{1} -1}  >d_{n_{2} -1}$.\\
It implies
\begin{align*}
g( y_{n_{1}} -d_{n_{2} -1}) &=g( y_{n_{2} -1})\text{ and}\\
g( y_{n_{1}} -d_{n_{2} -1} -\varepsilon ) &\neq g( y_{n_{2} -1} -\varepsilon )
\end{align*}
for sufficiently small $ \varepsilon  >0$.
Then the minimum preperiod is $ y_{n_{1}} -d_{n_{2} -1} =y_{n_{1}} -( y_{n_{2}} -y_{n_{2} -1})$ and it is a linear combination of removable numbers, but some coefficients may be negative. If so, then $ y_{n_{1}} -d_{n_{2} -1}$ is not a break.

\end{rem}

\begin{ques}

Is there any example of the preperiod not being equal to a break?

\end{ques}

Here is a final preparation before the results.

\begin{defntt}\label{dntt:diag}
\ 
\begin{enumerate}
\item \label{enum:per}To denote a periodic upper semi-continuous step function $f$. When we use the diagram
\begin{equation*}
\begin{tikzpicture}
\node at (-.2,1.72) [left] {$f(x)$};
\node at (-.2,2.22) [left] {$x\ $};

\cod{0}{y_0}
\cod{\exw}{y_1}
\cod{2*\exw}{y_2}
\cod{3*\exw}{y_3}

\cod{3.5*\exw}{y_{s-1}}
\cod{4.5*\exw}{y_s}
\cod{5.5*\exw}{x_0}
\cod{6.5*\exw}{y_{s+1}}
\cod{7.5*\exw}{y_{s+2}}
\cod{8*\exw}{y_t}
\cod{9*\exw}{x_1}

\gbox{1}{0}{\exw}{z_0}
\gbox{1}{\exw}{2*\exw}{z_1}
\gbox{1}{2*\exw}{3*\exw}{z_2}
\gbox{1}{3*\exw}{3.5*\exw}{\cdots}
\gbox{1}{3.5*\exw}{4.5*\exw}{z_{s-1}}
\gbox{1}{4.5*\exw}{5.5*\exw}{z_s}
\gbox{1}{5.5*\exw}{6.5*\exw}{z_s} \lper{1}{5.5*\exw}
\gbox{1}{6.5*\exw}{7.5*\exw}{z_{s+1}}
\gbox{1}{7.5*\exw}{8*\exw}{\cdots}
\gbox{1}{8*\exw}{9*\exw}{z_t} \rper{1}{9*\exw}
\end{tikzpicture}
\end{equation*}
with some $ y_{0} \leq y_{1} \leq y_{2} \leq \cdots \leq y_{s-1} \leq y_{s} \leq x_0 < y_{s+1} \leq \cdots \leq y_{t} \leq x_{1}$ :real numbers and $ z_{0} ,z_{1} ,\ldots ,z_{t} \in \mathbb{Z}_{\geq 0}$, it expresses that the function $ f:\mathbb{R}_{\geq y_{0}}\rightarrow \mathbb{Z}_{\geq 0}$ satisfies
\begin{equation*}
f( x) =\begin{cases}
z_{i} & ( y_{i} \leq x< y_{i+1} ;\ i=0,1,\ldots ,t-1)\\
z_{t} & ( y_{t} \leq x< x_{1})\\
f\bigl( x-( x_{1} -x_{0})\bigr) & ( x\geq x_{1})
\end{cases}.
\end{equation*}

We leave the possibility $y_i=y_{i+1}$ for generalization. If it holds, the corresponding box can be ignored.
The thick straight (resp. curved) line indicates the begining (resp. end) of the loop.
The corners of the box
\begin{equation*}
\begin{tikzpicture}
\cod{0}{y_i}
\cod{\exw}{y_{i+1}}
\gbox{1}{0}{\exw}{z_i}
\end{tikzpicture}
\end{equation*}
indicate that the interval $[y_i,y_{i+1})$ is left-closed (sharp corners) and right-open (rounded corners).\\

Here is an example of the diagram (recall Proposition~\ref{prop:ord}). We will explain the usage of the notation with it.
\begin{equation*}
\begin{tikzpicture}
\node at (-.2,0.22) [left] {$g_{\{2,4,7\}}(x-7)$};
\node at (-.2,0.72) [left] {$g_{\{2,4,7\}}(x-4)$};
\node at (-.2,1.22) [left] {$g_{\{2,4,7\}}(x-2)$};
\node at (-.2,1.72) [left] {$g_{\{2,4,7\}}(x)$};
\node at (-.2,2.22) [left] {$x\ $};

\cod{0}{0}
\cod{2*\egw}{2}
\cod{4*\egw}{4}
\cod{6*\egw}{6}
\cod{7*\egw}{7}
\cod{8*\egw}{8}
\cod{9*\egw}{9}
\cod{10*\egw}{10}
\cod{11*\egw}{11}
\cod{12*\egw}{12}
\cod{13*\egw}{13}
\cod{14*\egw}{14}
\cod{15*\egw}{15}
\cod{16*\egw}{16}
\cod{17*\egw}{17}
\cod{18*\egw}{18}

\gbox{1}{0}{2*\egw}{0}
\gbox{1}{2*\egw}{4*\egw}{1}
\gbox{1}{4*\egw}{6*\egw}{2}
\gbox{1}{6*\egw}{7*\egw}{0}
\gbox{1}{7*\egw}{8*\egw}{3}
\gbox{1}{8*\egw}{9*\egw}{1} \lper{1}{8*\egw}
\gbox{1}{9*\egw}{10*\egw}{0}
\gbox{1}{10*\egw}{11*\egw}{2} \rper{1}{11*\egw}
\gbox{1}{11*\egw}{12*\egw}{1} \lper{1}{11*\egw}
\gbox{1}{12*\egw}{13*\egw}{0}
\gbox{1}{13*\egw}{14*\egw}{2} \rper{1}{14*\egw}
\gbox{1}{14*\egw}{15*\egw}{1} \lper{1}{14*\egw}
\gbox{1}{15*\egw}{16*\egw}{0}
\gbox{1}{16*\egw}{17*\egw}{2} \rper{1}{17*\egw}
\gbox{1}{17*\egw}{18*\egw}{1} \lper{1}{17*\egw}
\gbox{2}{2*\egw}{4*\egw}{0}
\gbox{2}{4*\egw}{6*\egw}{1}
\gbox{2}{6*\egw}{8*\egw}{2}
\gbox{2}{8*\egw}{9*\egw}{0}
\gbox{2}{9*\egw}{10*\egw}{3}
\gbox{2}{10*\egw}{11*\egw}{1} \lper{2}{10*\egw}
\gbox{2}{11*\egw}{12*\egw}{0}
\gbox{2}{12*\egw}{13*\egw}{2} \rper{2}{13*\egw}
\gbox{2}{13*\egw}{14*\egw}{1} \lper{2}{13*\egw}
\gbox{2}{14*\egw}{15*\egw}{0}
\gbox{2}{15*\egw}{16*\egw}{2} \rper{2}{16*\egw}
\gbox{2}{16*\egw}{17*\egw}{1} \lper{2}{16*\egw}
\gbox{2}{17*\egw}{18*\egw}{0}
\gbox{3}{4*\egw}{6*\egw}{0}
\gbox{3}{6*\egw}{8*\egw}{1}
\gbox{3}{8*\egw}{10*\egw}{2}
\gbox{3}{10*\egw}{11*\egw}{0}
\gbox{3}{11*\egw}{12*\egw}{3}
\gbox{3}{12*\egw}{13*\egw}{1} \lper{3}{12*\egw}
\gbox{3}{13*\egw}{14*\egw}{0}
\gbox{3}{14*\egw}{15*\egw}{2} \rper{3}{15*\egw}
\gbox{3}{15*\egw}{16*\egw}{1} \lper{3}{15*\egw}
\gbox{3}{16*\egw}{17*\egw}{0}
\gbox{3}{17*\egw}{18*\egw}{2} \rper{3}{18*\egw}
\gbox{4}{7*\egw}{9*\egw}{0}
\gbox{4}{9*\egw}{11*\egw}{1}
\gbox{4}{11*\egw}{13*\egw}{2}
\gbox{4}{13*\egw}{14*\egw}{0}
\gbox{4}{14*\egw}{15*\egw}{3}
\gbox{4}{15*\egw}{16*\egw}{1} \lper{4}{15*\egw}
\gbox{4}{16*\egw}{17*\egw}{0}
\gbox{4}{17*\egw}{18*\egw}{2} \rper{4}{18*\egw}
\end{tikzpicture}
\end{equation*}
\begin{enumerate}
\item We may stack multiple rows of boxes with a single coordinate axis, whose domains may be different (so you can check that each value in the uppermost row is the $\mex$ of the values below).
\item \label{enum:num}We may split the diagram into several ones (in that case we may put the diagram numbers in angle brackets $\langle\ \rangle$).
\item We may continue the diagram after the thick curved line (the end of the loop).
\end{enumerate}
\item To denote a finite reputation appearing in a diagram as (1). We have
\begin{equation*}
\begin{tikzpicture}[baseline=1.65cm]
\node at (-.2,1.72) [left] {$g_{\{2,4,13\}}(x)$};
\node at (-.2,2.22) [left] {$x\ $};

\cod{0}{0}
\cod{2*\egw}{2}
\cod{4*\egw}{4}
\cod{6*\egw}{6}
\cod{8*\egw}{8}
\cod{10*\egw}{10}
\cod{12*\egw}{12}
\cod{13*\egw}{13}
\cod{14*\egw}{14}
\cod{15*\egw}{15}
\cod{16*\egw}{16}
\cod{17*\egw}{17}

\gbox{1}{0}{2*\egw}{0}
\gbox{1}{2*\egw}{4*\egw}{1}
\gbox{1}{4*\egw}{6*\egw}{2}
\gbox{1}{6*\egw}{8*\egw}{0}
\gbox{1}{8*\egw}{10*\egw}{1}
\gbox{1}{10*\egw}{12*\egw}{2}
\gbox{1}{12*\egw}{13*\egw}{0}
\gbox{1}{13*\egw}{14*\egw}{3}
\gbox{1}{14*\egw}{15*\egw}{1} \lper{1}{14*\egw}
\gbox{1}{15*\egw}{16*\egw}{0}
\gbox{1}{16*\egw}{17*\egw}{2} \rper{1}{17*\egw}

\end{tikzpicture},
\end{equation*}
and you can find that
\begin{equation*}
\begin{tikzpicture}
\gbox{1}{0}{2*\egw}{0}
\gbox{1}{2*\egw}{4*\egw}{1}
\gbox{1}{4*\egw}{6*\egw}{2}
\end{tikzpicture}
\end{equation*}
appears twice in a row. Then we write
\begin{equation*}
\begin{tikzpicture}[baseline=1.65cm]
\node at (-.2,1.72) [left] {$g_{\{2,4,13\}}(x)$};
\node at (-.2,2.22) [left] {$x\ $};

\cod{0}{0}
\cod{2*\egw}{2}
\cod{4*\egw}{4}
\cod{6*\egw}{6}

\cod{6*\egw+\repw}{12}
\cod{7*\egw+\repw}{13}
\cod{8*\egw+\repw}{14}
\cod{9*\egw+\repw}{15}
\cod{10*\egw+\repw}{16}
\cod{11*\egw+\repw}{17}

\gbox{1}{0}{2*\egw}{0}
\gbox{1}{2*\egw}{4*\egw}{1}
\gbox{1}{4*\egw}{6*\egw}{2}

\node at (6*\egw-0.1,1.6) [right] {$\times2$};

\gbox{1}{6*\egw+\repw}{7*\egw+\repw}{0}
\gbox{1}{7*\egw+\repw}{8*\egw+\repw}{3}
\gbox{1}{8*\egw+\repw}{9*\egw+\repw}{1} \lper{1}{8*\egw+\repw}
\gbox{1}{9*\egw+\repw}{10*\egw+\repw}{0}
\gbox{1}{10*\egw+\repw}{11*\egw+\repw}{2} \rper{1}{11*\egw+\repw}
\end{tikzpicture}.
\end{equation*}
In general, when we write
\begin{equation*}
\begin{tikzpicture}

\cod{0}{y_u}
\cod{2}{y_{u+1}}
\cod{4}{y_{u+2}}
\cod{5}{y_{u+p'-1}}
\cod{7}{y_{u+p'}}

\gbox{1}{0}{2}{z_u}
\gbox{1}{2}{4}{z_{u+1}}
\gbox{1}{4}{5}{\cdots}
\gbox{1}{5}{7}{z_{u+p'-1}}

\node at (6.9,1.6) [right] {$\times n$};
\end{tikzpicture}
\end{equation*}
with some $ u\in \mathbb{Z}_{\geq 0}$, $ p'\in \mathbb{Z}_{ >0}$, $ z_{u} ,z_{u+1} ,\ldots ,z_{u+p'} \in \mathbb{Z}_{\geq 0}$ and $ n\in \mathbb{Z}_{\geq 2}$ in the notation (1), it stands for
\begin{equation*}
\begin{tikzpicture}[baseline=1.65cm]
\cod{0}{y_u}
\cod{5.25}{y_{u+p'}\quad}
\cod{6.05}{\qquad\quad y_{u+(n-1)p'}}
\cod{11.3}{y_{u+np'}}
\gbox{1}{0}{1.5}{z_u}
\gbox{1}{1.5}{3}{z_{u+1}}
\gbox{1}{3}{3.75}{\cdots}
\gbox{1}{3.75}{5.25}{z_{u+p'-1}}
\node at (5.65,1.75) {$\cdots$};
\gbox{1}{6.05}{7.55}{z_u}
\gbox{1}{7.55}{9.05}{z_{u+1}}
\gbox{1}{9.05}{9.8}{\cdots}
\gbox{1}{9.8}{11.3}{z_{u+p'-1}}
\node at (5.8,1.5) [below] {$\underbrace{\hspace{11.6cm}}_{n\text{ times}}$};
\end{tikzpicture}.
\end{equation*}
\end{enumerate}

\end{defntt}

\section{Results}

Before the results, we note the multiple game so that the parameter can be bounded.

\begin{definition}

For $ G(S)$ and $ t\in \mathbb{R}_{ >0}$, let $tS\coloneqq \{ts\mid s\in S\}$, then the $ t$-multiple game of $ G(S)$ is the Continim $ G(tS)$.

\end{definition}

\begin{lema}\label{lem:mult}

Let $T\coloneqq tS$ with $t\in\mathbb{R}_{>0}$, then we have $g_S(x)=g_T(tx)\ (x\geq 0) .$

\end{lema}

\begin{proof}

By the definition of the Nim value,
\begin{align*}
g_T(tx)&=\mex\{g_T(x-\sigma)\mid \sigma\in T_{\leq tx}\}\\
&=\mex\{g_T(tx-ts)\mid s\in S_{\leq x}\}\\
&=\mex\{g_T\bigl(t(x-s)\bigr)\mid s\in S_{\leq x}\}\ (x\geq 0).
\end{align*}
Thus the function $f(x)\coloneqq g_T(tx)$ satisfies
\begin{equation*}
f(x)=\mex\{f(x-s)\mid s\in S_{\leq x}\}\ (x\geq 0).
\end{equation*}
By Proposition~\ref{prop:fg}, we get $f=g_S$.

\end{proof}

\begin{conv}

For $G(S)$ with $S=\{s_1,s_2,\ldots,s_{r-1},s_r\}$, $0<s_1<s_2<\cdots<s_{r-1}<s_r$, let
\begin{equation*}
S'=\frac{S}{s_r}=\left\{\frac{s_1}{s_r},\frac{s_2}{s_r},\ldots,\frac{s_{r-1}}{s_r},1\right\},
\end{equation*}
then $ \max S'=1$ and $ S=s_rS'$, so the Nim value function of $ G(S)$ can be reconstructed by that of $ G( S')$ using Lemma~\ref{lem:mult}. In the following, unless otherwise stated, we assume that $\max S=1$ (bounded).

\end{conv}

\begin{rem}

Our results below apply for the classical Subtraction Nim with the set of removable numbers $\{a,b,c\}$ where $a<b<c$, by replacing $S$ with $\{\frac{a}{c},\frac{b}{c},1\}$.

\end{rem}

\begin{thm}\label{th:r001}

For $0<s_1<s_2<\cdots<s_{r-1}<s_r=1$, assume that $s_{j+1}\leq s_j+s_1$ holds for all $j$ with $1\leq j< r$. Then the Nim value function $g(x)$ of $G(\{s_1,s_2,\ldots,s_r\})$ agrees with
\begin{equation*}
f(x) \coloneqq \left\lfloor \displaystyle\frac{x\bmod( 1+s_{1})}{s_{1}}\right\rfloor.
\end{equation*}
In particular $g$ is purely periodic with the minimum period $1+s_1$.

\end{thm}

\begin{proof}

We prove that $f(x)$ agrees with $g(x)$ by using Proposition~\ref{prop:fg}. For $x<s_1$, it is a terminal position and we have 
\begin{equation*}
f(x)=0=\mex \emptyset= \{f(x-s)\mid x\in S_{\leq x}\}.
\end{equation*}
For $x\in[s_1,1+s_1)$, we have $f(x)=\displaystyle\left\lfloor\frac{x}{s_1}\right\rfloor$. Let $\rho\coloneqq\#\{i\leq r\mid s_i\leq x\}$, then the sequence
\begin{equation*}
f(x-s_1),f(x-s_2),\ldots,f(x-s_\rho)
\end{equation*}
satisfies following:
\begin{enumerate}
\item[$(1)_0$]
$f(x-s_1)=f(x)-1$.
\item[$(2)_0$]
We show that $f(x-s_{j+1})-f(x-s_j)=0\text{ or }-1\ (j=1,2,\ldots,\rho-1)$. We have
\begin{align*}
f(x-s_{j+1})&=\left\lfloor\frac{x-s_{j+1}}{s_1}\right\rfloor\\
&=\left\lfloor\frac{(x-s_j)-(s_{j+1}-s_j)}{s_1}\right\rfloor\\
&=\left\lfloor\frac{x-s_j}{s_1}-\frac{s_{j+1}-s_j}{s_i}\right\rfloor\\
&=f(x-s_j)\text{ or }f(x-s_j)-1,
\end{align*}
because $0<\displaystyle\frac{s_{j+1}-s_j}{s_i}\leq 1$ holds.
\item[$(3)_0$]
We show that $f(x-s_\rho)=0$. In fact, if $s_1\leq x<1$, then we get $s_\rho\leq x<s_{\rho+1}$ and 
\begin{equation*}
x-s_\rho<s_{\rho+1}-s_\rho\leq s_1.
\end{equation*}
If $1\leq x<1+s_1$, then we get $\rho=r$ and 
\begin{equation*}
x-s_\rho=x-1<s_1.
\end{equation*}
\end{enumerate}
Thus they form the set $\{0,1,\ldots,f(x)-1\}$ and $f(x)$ is the $\mex$ of them. We have shown that 
\begin{equation*}
f(x)=\mex\{f(x-s_i)\mid s_i\leq x\}\ (x\in[s_1,1+s_1)).
\end{equation*}
Likewise, for $x\in\bigl[k(1+s_1),(k+1)(1+s_1)\bigr)\ (k=1,2,\ldots)$, we have $f(x)=\displaystyle\left\lfloor\frac{x-k(1+s_1)}{s_1}\right\rfloor$. Let $\rho\coloneqq\#\{i\leq r\mid s_i\leq x-k(1+s_1)\}$, then we have
\begin{equation*}
x-s_i\in
\begin{cases}
\bigl[k(1+s_1),(k+1)(1+s_1)\bigr) &(i=1,2,\ldots,\rho)\\
\bigl[(k-1)(1+s_1),k(1+s_1)\bigr) &(i=\rho+1,\rho+2,\ldots,r)
\end{cases}.
\end{equation*}
If $x\geq k(1+s_1)+s_1$, then $\rho\geq 1$ and the sequence
\begin{equation*}
f(x-s_1),f(x-s_2),\ldots,f(x-s_\rho)
\end{equation*}
satisfies following:
\begin{enumerate}
\item[$(1)_k$]
$f(x-s_1)=f(x)-1$.\label{enum:f-1}
\item[$(2)_k$]
$f(x-s_{j+1})-f(x-s_j)=0\text{ or }-1\ (j=1,2,\ldots,\rho-1)$, because\label{enum:0-1}
\begin{align*}
 f(x-s_{j+1})
&=\left\lfloor\frac{x-s_{j+1}-k(1+s_1)}{s_1}\right\rfloor\\
&=\left\lfloor\frac{\bigl(x-s_j-k(1+s_1)\bigr)-(s_{j+1}-s_j)}{s_1}\right\rfloor\\
&=\left\lfloor\frac{x-s_j-k(1+s_1)}{s_1}-\frac{s_{j+1}-s_j}{s_1}\right\rfloor\\
&=f(x-s_j)\text{ or }f(x-s_j)-1.
\end{align*}
\item[$(3)_k$] We have $f(x-s_\rho)=0$, because if $k(1+s_1)\leq x<k(1+s_1)+1$, then we get $s_\rho\leq x-k(1+s_1)<s_{\rho+1}$ and 
\begin{equation*}
x-k(1+s_1)-s_\rho<s_{\rho+1}-s_\rho\leq s_1.
\end{equation*}
If $k(1+s_1)+1\leq x<(k+1)(1+s_1)$, then we get $\rho=r$ and 
\begin{equation*}
x-k(1+s_1)-s_\rho=x-k(1+s_1)-1<s_1.
\end{equation*}
\end{enumerate}
If $x< k(1+s_1)+1$, then $\rho<r$ and the sequence
\begin{equation*}
f(x-s_{\rho+1}),f(x-s_{\rho+2}),\ldots,f(x-s_r)
\end{equation*}
satisfies following:
\begin{enumerate}
\item[$(2)'_k$]
$f(x-s_{j+1})-f(x-s_j)=0\text{ or }-1\ (j=\rho,\rho+1,\ldots,r-1)$, because
\begin{align*}
f(x-s_{j+1})&=\left\lfloor\frac{x-s_{j+1}-(k-1)(1+s_1)}{s_1}\right\rfloor\\
&=\left\lfloor\frac{\bigl(x-s_j-(k-1)(1+s_1)\bigr)-(s_{j+1}-s_j)}{s_1}\right\rfloor\\
&=\left\lfloor\frac{x-s_j-(k-1)(1+s_1)}{s_1}-\frac{s_{j+1}-s_j}{s_1}\right\rfloor\\
&=f(x-s_j)\text{ or }f(x-s_j)-1.
\end{align*}
\item[$(4)_k$]
We have $f(x-s_r)=f(x-1)=f(x)+1$, because
\begin{align*}
f(x-1)&=\left\lfloor\frac{x-1-(k-1)(1+s_1)}{s_1}\right\rfloor\\
&=\left\lfloor\frac{x-k(1+s_1)+s_1}{s_1}\right\rfloor\\
&=\left\lfloor\frac{x-k(1+s_1)}{s_1}+1\right\rfloor\\
&=f(x)+1.
\end{align*}
\end{enumerate}
Thus the terms of the first sequence form the set $\{0,1,\ldots,f(x)-1\}$ and those of the second are greater than $f(x)$. Hence $f(x)$ is the $\mex$ of all of them. We have shown that 
\begin{equation*}
f(x)=\mex\{f(x-s_i)\mid s_i\leq x\}\ (x\in\bigl[k(1+s_1),(k+1)(1+s_1)\bigr),\ k=1,2,\ldots).
\end{equation*}

\end{proof}

\begin{thm}\label{th:necktie}

For $ 0< a,b< 1$ with $ b\leq 2a$ and $a\leq 2b$, the Nim value function of $ G(S)$ for $ S=\{a,b,1\}$ is purely periodic if $(a,b)$ is in any of the colored region (to be precisely defined later) on Fig.~\ref{fig:necktie}. By exchanging $a$ and $b$ if necessary, we assume $a\leq b$ in the following. Then for each colored region the minimum pure period is
\begin{equation*}
\begin{cases}
a+b & (\text{if }(a,b)\text{ is in a green region})\\
a+1 & (\text{if }(a,b)\text{ is in a red region})\\
b+1 & (\text{if }(a,b)\text{ is in a blue region})
\end{cases} ,
\end{equation*}
and if $(a,b)$ is in a white region, it is not purely periodic with the period $a+b$, $1+a$ nor $1+b$.

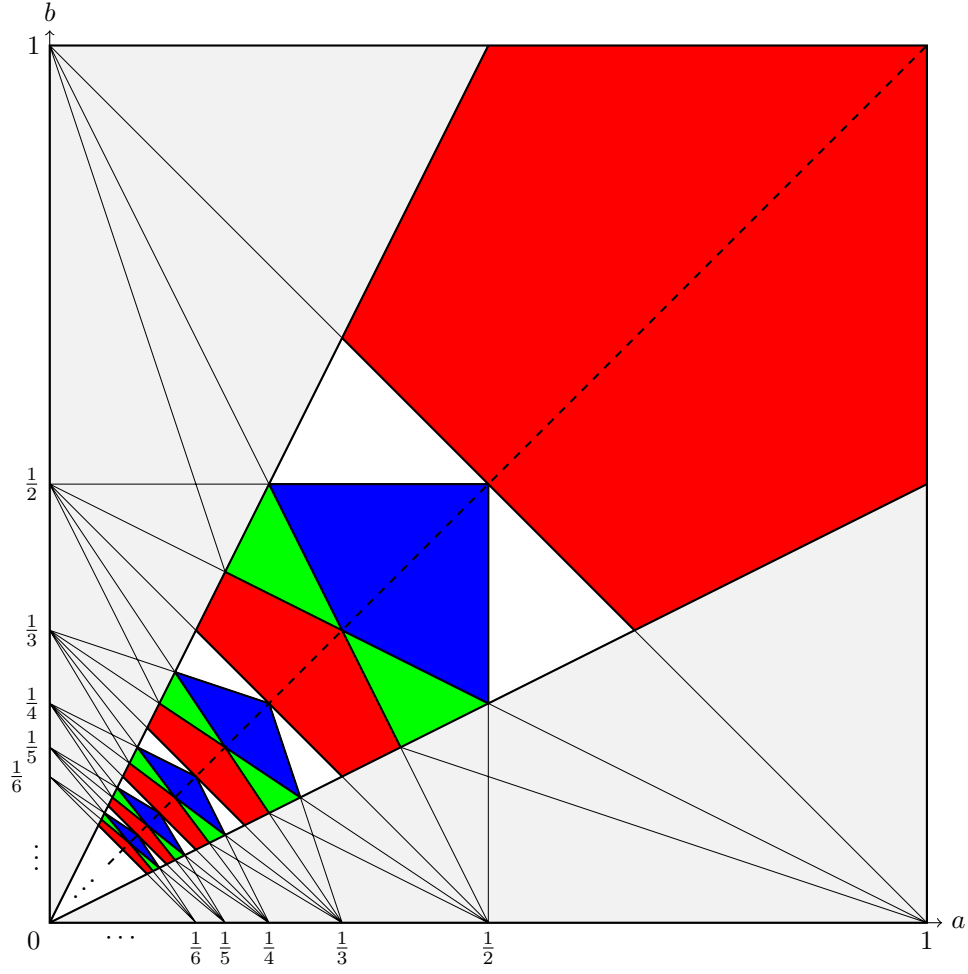
\begin{figure}[hbtp]  % Used: TikZ
\centering
\begin{tikzpicture}
\draw [ultra thin, fill=black!5, line join=bevel] (0,0) --(\fgx,0) --(\fgx,\fgx/2) --(0,0) --(0,\fgx) --(\fgx/2,\fgx) --(0,0);
\draw [very thin] (0,\fgx/2) --(\fgx/2,\fgx/2) --(\fgx/2,0)
foreach \j in {1,...,3}  {(0,\fgx) --(\fgx/\j,0)} 
foreach \j in {1,...,4}  {(0,\fgx/2) --(\fgx/\j,0)}
foreach \j in {1,...,5}  {(0,\fgx/3) --(\fgx/\j,0)}
foreach \j in {2,...,6}  {(0,\fgx/4) --(\fgx/\j,0)}
foreach \j in {3,...,6}  {(0,\fgx/5) --(\fgx/\j,0)}
foreach \j in {4,...,6}  {(0,\fgx/6) --(\fgx/\j,0)};
\filldraw [red] foreach \i in {0,...,5}
 {({divide(\fgx,3*\i+2)},{divide(2*\fgx,3*\i+2)})
  --({divide(\fgx,3*\i+3)},{divide(2*\fgx,3*\i+3)})
  --({divide(2*\fgx,3*\i+3)},{divide(\fgx,3*\i+3)})
  --({divide(2*\fgx,3*\i+2)},{divide(\fgx,3*\i+2)})
  --({divide(\fgx,2*\i+1)},{divide(\fgx,2*\i+1)}) --cycle};
\filldraw [green] foreach \i in {1,...,5}
 {({divide(\fgx,3*\i+1)},{divide(2*\fgx,3*\i+1)})
  --({divide(\fgx,3*\i+2)},{divide(2*\fgx,3*\i+2)})
  --({divide(2*\fgx,3*\i+1)},{divide(\fgx,3*\i+1)})
  --({divide(2*\fgx,3*\i+2)},{divide(\fgx,3*\i+2)}) --cycle};
\filldraw [blue, line join=bevel] foreach \i in {1,...,5}
 {({divide(\fgx,3*\i+1)},{divide(2*\fgx,3*\i+1)})
  --({divide(\fgx,2*\i+1)},{divide(\fgx,2*\i+1)})
  --({divide(2*\fgx,3*\i+1)},{divide(\fgx,3*\i+1)})
  --({divide(\fgx,2*\i)},{divide(\fgx,2*\i)}) --cycle};
\draw [thick] (0,0) rectangle(\fgx,\fgx);
\draw (0,0) [->] --(0,\fgx+.2);
\draw (0,0) [->] --(\fgx+.2,0);
\node at (\fgx+.2,0) [right] {$a$};
\node at (0,\fgx+.2) [above] {$b$};
\draw [thick, line join=bevel] (\fgx/2,\fgx) --(0,0) --(\fgx,\fgx/2)
(\fgx/3,2*\fgx/3) --(2*\fgx/3,\fgx/3) foreach \i in {1,...,5}
 {  ({divide(\fgx,3*\i+2)},{divide(2*\fgx,3*\i+2)})
  --({divide(2*\fgx,3*\i+1)},{divide(\fgx,3*\i+1)})
  --({divide(\fgx,2*\i)},{divide(\fgx,2*\i)})
  --({divide(\fgx,3*\i+1)},{divide(2*\fgx,3*\i+1)})
  --({divide(2*\fgx,3*\i+2)},{divide(\fgx,3*\i+2)})
    ({divide(\fgx,3*\i+3)},{divide(2*\fgx,3*\i+3)})
  --({divide(2*\fgx,3*\i+3)},{divide(\fgx,3*\i+3)})};
\draw [thick, dashed] (\fgx/15,\fgx/15) --(\fgx,\fgx);
\path node at (\fgx,0) [below] {$1$} 
      node at (0,\fgx) [left] {$1$}
foreach \i in {2,...,6} {node at (\fgx/\i,0) [below] {$\frac{1}{\i}$}}
foreach \i in {2,...,5} {node at (0,\fgx/\i) [left] {$\frac{1}{\i}$}} 
      node at (-.2,\fgx/6) [left] {$\frac{1}{6}$}
      node at (\fgx/12,0) [below] {$\cdots$}
      node at (0,\fgx/12) [left] {$\vdots$}
      node at (0,0) [below left] {$0$};
\path (0,-.02) -- node[sloped]{$\cdots$} (\fgx/12,\fgx/12-.02);
\end{tikzpicture}
\caption{The periods of the Nim value function of $G(S)$ where $S=\{a,b,1\}$ with $0<a,b<1$, $b\leq 2a$ and $a\leq 2b$.}\label{fig:necktie}
\end{figure}

Let $n$ be a nonnegative integer and assume that $na+nb<1\leq(n+1)a+(n+1)b$. Then each region is defined by
\begin{equation*}
\left\{
\begin{alignedat}{5}
(n+1)&a+&    n&b\leq1\leq&    n&a+&(n+1)&b &     &(n\text{th green region }G_n\ (n\geq 1))\\
    n&a+&(n+1)&b   <1\leq&(n+1)&a+&(n+1)&b &     &(n\text{th red region }  R_n\ (n\geq 0))\\
    n&a+&    n&b   <1   <&(n-1)&a+&(n+1)&b &     &(n\text{th white region }W_n\ (n\geq 1))\\
(n-1)&a+&(n+1)&b\leq1   <&(n+1)&a+&    n&b &\quad&(n\text{th blue region } B_n\ (n\geq 1))
\end{alignedat}
\right. .
\end{equation*}
In order to describe the Nim value functions precisely (and partially for $W_n$), we split each $R_n$ into two parts, each $W_n$ into four, and each $B_n$ into three, as shown in Fig.~\ref{fig:split}:
\begin{equation*}
\left\{
\begin{alignedat}{4}
    n&a+&(n+1)&b<1\leq(n+2)a+&    n&b &\quad& (R_n^1)\\
(n+2)&a+&    n&b<1\leq(n+1)a+&(n+1)&b &     & (R_n^2)
\end{alignedat}
\right. ,
\end{equation*}
\begin{equation*}
\left\{
\begin{alignedat}{4}
    n&a+&    n&b<1\leq&\min(&(n-1)a+(n+1)b,(n+3)a+(n-2)b) &\quad& (W_n^1)\\
(n+3)&a+&(n-2)&b<1   <&     &(n-1)a+(n+1)b                &     & (W_n^2)\\
    n&a+&    n&b<1\leq&     &(n+2)a+(n-1)b<(n-1)a+(n+1)b  &     & (W_n^3)\\
(n+2)&a+&(n-1)&b<1   <&     &(n-1)a+(n+1)b                &     & (W_n^4)
\end{alignedat}
\right. ,
\end{equation*}
\begin{equation*}
\left\{
\begin{alignedat}{4}
                             (n-1)a+(n+1)b& &&{}\leq1<(n+3)a+&(n-2)&b &\quad& (B_n^1)\\
\max\bigl((n-1)a+(n+1)b,(n+3)a+(n-2)b&\bigr)&&{}\leq1<(n+2)a+&(n-1)&b &     & (B_n^2)\\
\max\bigl((n-1)a+(n+1)b,(n+2)a+(n-1)b&\bigr)&&{}\leq1<(n+1)a+&    n&b &     & (B_n^3)
\end{alignedat}
\right. .
\end{equation*}

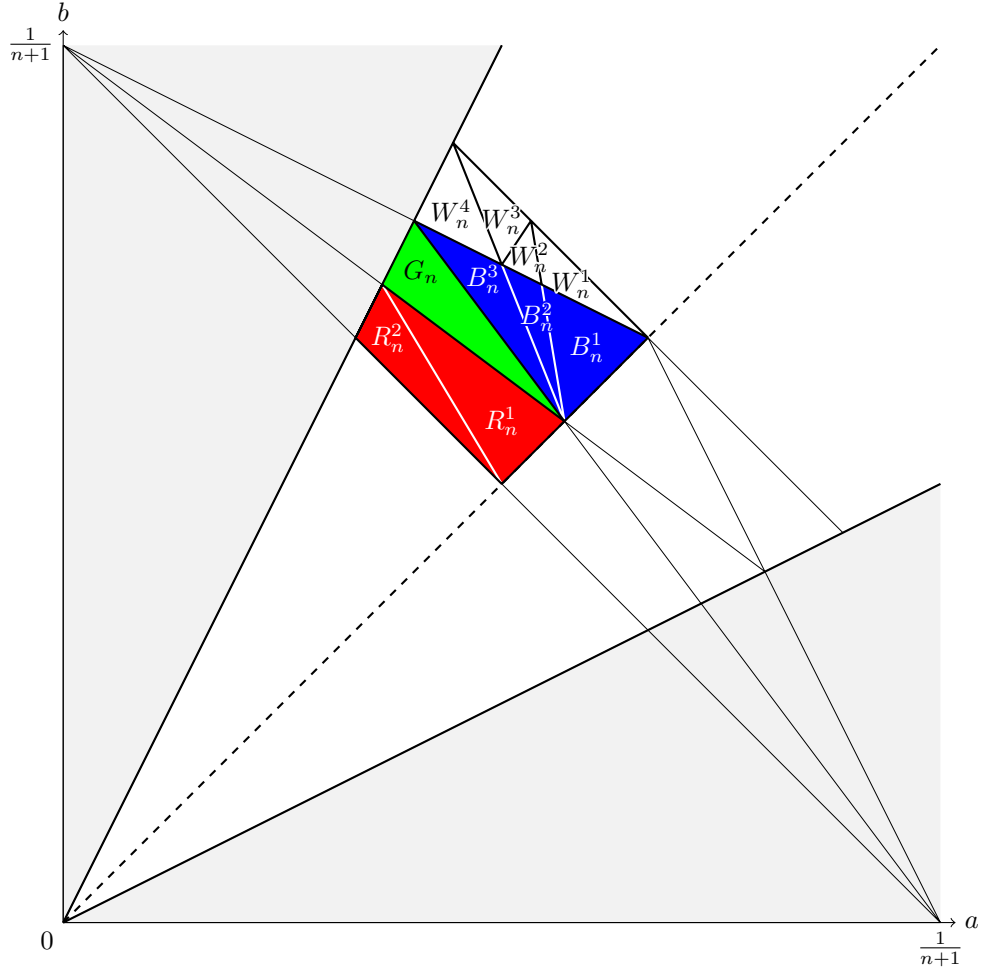
\begin{figure}[hbtp] % Used: TikZ
\centering
\begin{tikzpicture}
\draw [ultra thin, fill=black!5, line join=bevel] (\fgx,\fgx/2) --(0,0) --(\fgx,0);
\draw [ultra thin, fill=black!5, line join=bevel] (\fgx/2,\fgx) --(0,0) --(0,\fgx);
\draw [very thin]
foreach \j in {10,...,12}  {(0,\fgx) --(4*\fgx/\j,8*\fgx/\j)} 
foreach \j in {6,...,8}  {(\fgx,0) --(4*\fgx/\j,4*\fgx/\j)}
(4*\fgx/7,4*\fgx/7) --(8*\fgx/10,4*\fgx/10) (4*\fgx/6,4*\fgx/6) --(8*\fgx/9,4*\fgx/9);
\filldraw [red] (4*\fgx/7,4*\fgx/7) --(4*\fgx/11,8*\fgx/11) --(4*\fgx/12,8*\fgx/12) --(4*\fgx/8,4*\fgx/8);
\filldraw [green, line join=bevel] (4*\fgx/10,8*\fgx/10) --(4*\fgx/11,8*\fgx/11) --(4*\fgx/7,4*\fgx/7) --cycle;
\filldraw [blue, line join=bevel] (4*\fgx/6,4*\fgx/6) --(4*\fgx/10,8*\fgx/10) --(4*\fgx/7,4*\fgx/7);
\draw (0,0) [->] --(0,\fgx+.2);
\draw (0,0) [->] --(\fgx+.2,0);
\node at (\fgx+.2,0) [right] {$a$};
\node at (0,\fgx+.2) [above] {$b$};
\draw [white, thick, line join=bevel] (4*\fgx/8,4*\fgx/8) --(4*\fgx/11,8*\fgx/11)
(8*\fgx/16,12*\fgx/16) --(4*\fgx/7,4*\fgx/7) --(12*\fgx/22,16*\fgx/22);
\draw [thick, line join=bevel] (\fgx/2,\fgx) --(0,0) --(\fgx,\fgx/2)
(12*\fgx/22,16*\fgx/22) --(8*\fgx/15,12*\fgx/15) --(8*\fgx/16,12*\fgx/16)
-- (4*\fgx/9,8*\fgx/9) --(4*\fgx/6,4*\fgx/6) --(4*\fgx/10,8*\fgx/10) --(4*\fgx/7,4*\fgx/7) --(4*\fgx/11,8*\fgx/11) --(4*\fgx/12,8*\fgx/12) --(4*\fgx/8,4*\fgx/8) --(4*\fgx/6,4*\fgx/6);
\draw [thick, dashed] (0,0) --(\fgx,\fgx);
\path node at (\fgx,0) [below] {$\frac{1}{n+1}$} 
      node at (0,\fgx) [left] {$\frac{1}{n+1}$}
      node at (0,0) [below left] {$0$};
\draw (4*\fgx/6.9,4*\fgx/5.5) node {\contour{white}{$W_n^1$}};
\draw (4*\fgx/7.53,4*\fgx/5.28) node {\contour{white}{$W_n^2$}};
\draw (4*\fgx/8,4*\fgx/5) node {\contour{white}{$W_n^3$}};
\draw (4*\fgx/9.05,4*\fgx/4.95) node {$W_n^4$};
\draw [white] (4*\fgx/6.7,4*\fgx/6.1) node {$B_n^1$};
\draw [white] (4*\fgx/7.4,4*\fgx/5.8) node {\contour{blue}{$B_n^2$}};
\draw [white] (4*\fgx/8.35,4*\fgx/5.45) node {$B_n^3$};
\draw (4*\fgx/9.8,4*\fgx/5.39) node {$G_n$};
\draw [white] (4*\fgx/8,4*\fgx/7) node {$R_n^1$};
\draw [white] (4*\fgx/10.8,4*\fgx/6) node {$R_n^2$};
\end{tikzpicture}
\caption{Splitting $B_n$ into $B_n^1$, $B_n^2$ and $B_n^3$, $R_n$ into $R_n^1$ and $R_n^2$, and $W_n$ into $W_n^1$, $W_n^2$, $W_n^3$ and $W_n^4$, numbered from right. $G_n$ contains the borders shared with $R_n$ and $B_n$. Each of $R_n$ and $B_n$ contains the borders shared with $W_n$.}\label{fig:split}
\end{figure}

Let $m=a+b$, and let $\lambda$ be the ``remainder'' when you divide $1-a$ by $m(=a+b)$. In other words,
\begin{equation*}
\lambda=1-a-\left\lfloor\frac{1-a}{m}\right\rfloor m.
\end{equation*}
Then this $\lambda$ determines the ``color'' of $(a,b)$ as follows:
\begin{equation*}
\left\{
\begin{alignedat}{2}
               0 &{}\leq \lambda \leq     b-a       &     & (G_n  )\\
             b-a &{}   < \lambda \leq       a       &     & (R_n^1)\\
               a &{}   < \lambda \leq       b       &     & (R_n^2)\\
               b &{}   < \lambda <  \min(2b-a,3a-b) &     & (W_n^1)\\
            3a-b &{}\leq \lambda <       2b-a       &     & (W_n^2)\\
               b &{}   < \lambda <         2a<2b-a  &     & (W_n^3)\\
              2a &{}\leq \lambda <       2b-a       &     & (W_n^4)\\
            2b-a &{}\leq \lambda <       3a-b       &\quad& (B_n^1)\\
\max( 2b-a,3a-b) &{}\leq \lambda <         2a       &     & (B_n^2)\\
  \max( 2b-a,2a) &{}\leq \lambda <        a+b=m     &     & (B_n^3)
\end{alignedat}
\right. .
\end{equation*}

\end{thm}

\begin{rem}

We conjecture that the Nim value function for $W_n$ is not purely periodic with any period, and that each $W_n$ is split into infinitely many parts.

\end{rem}

\begin{rem}

$B_n^1$ and $R_{n-1}^2$ share the point $(a,b)= \displaystyle\left(\frac{1}{2n} ,\frac{1}{2n}\right)$ (there is no other overlap). We can check their Nim value function (as below) are the same. See Appendix~\ref{appb}. Note that the point satisfies $a=b<1$ and $|\{a,b,1\}|=2$. Our result includes that case. See Corollary~\ref{cor:a1} for details. \label{nt:share}

\end{rem}

\begin{landscape}
\begin{proof}
\ 
Throughout this proof, and in Remarks~\ref{rem:border} and \ref{rem:overlap}, we use the notation $m\coloneqq a+b$. Define $ f:\mathbb{R}_{\geq 0}\rightarrow \mathbb{Z}_{\geq 0}$ for $B_n$, $G_n$ and $R_n$ as
\begin{equation*}
\left\{
\begin{alignedat}{2}
&\begin{tikzpicture}[baseline=1.65cm]
\cod{0}{0} \cod{1.5}{a} \cod{3}{2a} \cod{19/8*1.5}{m}
\cod{19/8*1.5+\repw}{nm} \cod{\hcc}{1}
\cod{\hcc+1/8*\haa}{\qquad\quad(n-1)m+3a} \cod{\hcc+4/8*\haa}{nm+a}
\cod{\hcc+\haa}{1+a}
\cod{\hcc+11/8*\haa}{1+b}
\gbox{1}{0}{1.5}{0} \lper{1}{0}
\gbox{1}{1.5}{3}{1} \gbox{1}{3}{19/8*1.5}{2}
\node at (19/8*1.5-0.1,1.6) [right] {$\times n$};

\gbox{1}{19/8*1.5+\repw}{\hcc}{0}
\gbox{1}{\hcc}{\hcc+1/8*\haa}{2}
\gbox{1}{\hcc+1/8*\haa}{\hcc+4/8*\haa}{3}
\gbox{1}{\hcc+4/8*\haa}{\hcc+\haa}{1}
\gbox{1}{\hcc+\haa}{\hcc+11/8*\haa}{3} \rper{1}{\hcc+11/8*\haa}
\end{tikzpicture} &&                                          (B_n^1)
\\
&\begin{tikzpicture}[baseline=1.65cm]
\cod{0}{0} \cod{1.5}{a} \cod{3}{2a} \cod{19/8*1.5}{m}
\cod{19/8*1.5+\repw}{nm} \cod{\hcc}{1}
\cod{\hcc+1/8*\haa}{\qquad\quad(n-1)m+3a} \cod{\hcc+4/8*\haa}{nm+a}
\cod{\hcc+\haa}{1+a}
\cod{\hcc+9/8*\haa}{\begin{tabular}{c}$(n-1)m+4a$\\$\downarrow$\end{tabular}}
\cod{\hcc+11/8*\haa}{1+b}
\gbox{1}{0}{1.5}{0} \lper{1}{0}
\gbox{1}{1.5}{3}{1} \gbox{1}{3}{19/8*1.5}{2}
\node at (19/8*1.5-0.1,1.6) [right] {$\times n$};

\gbox{1}{19/8*1.5+\repw}{\hcc}{0}
\gbox{1}{\hcc}{\hcc+1/8*\haa}{2}
\gbox{1}{\hcc+1/8*\haa}{\hcc+4/8*\haa}{3}
\gbox{1}{\hcc+4/8*\haa}{\hcc+\haa}{1}
\gbox{1}{\hcc+\haa}{\hcc+9/8*\haa}{3}
\gbox{1}{\hcc+9/8*\haa}{\hcc+11/8*\haa}{2} \rper{1}{\hcc+11/8*\haa}
\end{tikzpicture} &&                                          (B_n^2)
\\
&\begin{tikzpicture}[baseline=1.65cm]
\cod{0}{0} \cod{1.5}{a} \cod{3}{2a} \cod{19/8*1.5}{m}
\cod{19/8*1.5+\repw}{nm} \cod{\hcc}{1}
\cod{\hcc+4/8*\haa}{nm+a}
\cod{\hcc+\haa}{1+a}
\cod{\hcc+11/8*\haa}{1+b}
\gbox{1}{0}{1.5}{0} \lper{1}{0}
\gbox{1}{1.5}{3}{1} \gbox{1}{3}{19/8*1.5}{2}
\node at (19/8*1.5-0.1,1.6) [right] {$\times n$};

\gbox{1}{19/8*1.5+\repw}{\hcc}{0}
\gbox{1}{\hcc}{\hcc+4/8*\haa}{3}
\gbox{1}{\hcc+4/8*\haa}{\hcc+\haa}{1}
\gbox{1}{\hcc+\haa}{\hcc+11/8*\haa}{2} \rper{1}{\hcc+11/8*\haa}
\end{tikzpicture} &&                                          (B_n^3)
\\
&\begin{tikzpicture}[baseline=1.65cm]
\cod{0}{0} \cod{1.5}{a} \cod{3}{2a} \cod{19/8*1.5}{m}
\gbox{1}{0}{1.5}{0} \lper{1}{0}
\gbox{1}{1.5}{3}{1} \gbox{1}{3}{19/8*1.5}{2} \rper{1}{19/8*1.5}
\end{tikzpicture} &&                                            (G_n)
\\
&\begin{tikzpicture}[baseline=1.65cm]
\cod{0}{0} \cod{1.5}{a} \cod{3}{2a} \cod{19/8*1.5}{m}
\cod{19/8*1.5+\repw}{nm}
\cod{\hcc}{nm+a}
\cod{\hcc+4/8*\haa}{nm+2a}
\cod{\hcc+10/8*\haa}{1+a}
\gbox{1}{0}{1.5}{0} \lper{1}{0}
\gbox{1}{1.5}{3}{1} \gbox{1}{3}{19/8*1.5}{2}
\node at (19/8*1.5-0.1,1.6) [right] {$\times n$};

\gbox{1}{19/8*1.5+\repw}{\hcc}{0}
\gbox{1}{\hcc}{\hcc+4/8*\haa}{1}
\gbox{1}{\hcc+4/8*\haa}{\hcc+10/8*\haa}{2} \rper{1}{\hcc+10/8*\haa}
\end{tikzpicture} &&                                          (R_n^1)
\\
&\begin{tikzpicture}[baseline=1.65cm]
\cod{0}{0} \cod{1.5}{a} \cod{3}{2a} \cod{19/8*1.5}{m}
\cod{19/8*1.5+\repw}{nm}
\cod{\hcc}{nm+a}
\cod{\hcc+4/8*\haa}{nm+2a}
\cod{\hcc+\haa}{nm+3a}
\cod{\hcc+10/8*\haa}{1+a}
\gbox{1}{0}{1.5}{0} \lper{1}{0}
\gbox{1}{1.5}{3}{1} \gbox{1}{3}{19/8*1.5}{2}
\node at (19/8*1.5-0.1,1.6) [right] {$\times n$};

\gbox{1}{19/8*1.5+\repw}{\hcc}{0}
\gbox{1}{\hcc}{\hcc+4/8*\haa}{1}
\gbox{1}{\hcc+4/8*\haa}{\hcc+\haa}{2}
\gbox{1}{\hcc+\haa}{\hcc+10/8*\haa}{3} \rper{1}{\hcc+10/8*\haa}
\end{tikzpicture} &&                                          (R_n^2)
\end{alignedat}
\right.
\end{equation*}
and $\varphi:[0,1+a+b)\rightarrow \mathbb{Z}_{\geq 0}$ for $W_n$ as
\begin{equation*}
\left\{
\begin{alignedat}{2}
&\begin{tikzpicture}[baseline=1.65cm]
\cod{-\repw-19/8*1.5}{0}
\cod{-\repw-19/8*1.5+1.5}{a}
\cod{-\repw-19/8*1.5+3}{2a}
\cod{-\repw}{m}
\cod{0}{nm}
\cod{\aaa-\bbc/2}{1}
\cod{2*\aaa-\bbc}{(n-1)m+3a}
\cod{\aaa}{nm+a}
\cod{2*\aaa-\bbc/2}{\begin{tabular}{c}$1+a$\\$\downarrow$\end{tabular}}
\cod{\bbc}{\quad\  nm+b}
\cod{\aaa+\bbc/2}{1+b}
\cod{3*\aaa-\bbc/2}{1+2a\qquad}
\cod{\mmc}{\begin{tabular}{c}$(n+1)m$\\$\downarrow$\end{tabular}}
\cod{2*\aaa+\bbc/2}{1+a+b}
\gbox{1}{-\repw-19/8*1.5}{-\repw-19/8*1.5+1.5}{0}
\gbox{1}{-\repw-19/8*1.5+1.5}{-\repw-19/8*1.5+3}{1}
\gbox{1}{-\repw-19/8*1.5+3}{-\repw}{2}
\node at (-\repw-0.1,1.6) [right] {$\times n$};
\gbox{1}{0}{\aaa-\bbc/2}{0}
\gbox{1}{\aaa-\bbc/2}{2*\aaa-\bbc}{2}
\gbox{1}{2*\aaa-\bbc}{\aaa}{3}
\gbox{1}{\aaa}{2*\aaa-\bbc/2}{1}
\gbox{1}{2*\aaa-\bbc/2}{\bbc}{0}
\gbox{1}{\bbc}{\aaa+\bbc/2}{3}
\gbox{1}{\aaa+\bbc/2}{3*\aaa-\bbc/2}{0}
\gbox{1}{3*\aaa-\bbc/2}{\mmc}{1}
\gbox{1}{\mmc}{2*\aaa+\bbc/2}{0}
\end{tikzpicture} &&                                          (W_n^1)\\
&\begin{tikzpicture}[baseline=1.65cm]
\cod{-\repw-19/8*1.5}{0}
\cod{-\repw-19/8*1.5+1.5}{a}
\cod{-\repw-19/8*1.5+3}{2a}
\cod{-\repw}{m}
\cod{0}{nm}
\cod{\aaa-\bbc/2}{1}
\cod{2*\aaa-\bbc}{(n-1)m+3a}
\cod{\aaa}{nm+a}
\cod{2*\aaa-\bbc/2}{\begin{tabular}{c}$1+a$\ \\$\downarrow$\end{tabular}}
\cod{\bbc}{\quad\  nm+b}
\cod{3*\aaa-\bbc}{\begin{tabular}{c}\qquad$(n-1)m+4a$\quad\\$\downarrow$\end{tabular}}
\cod{\aaa+\bbc/2}{\quad 1+b}
\cod{3*\aaa-\bbc/2}{1+2a\qquad}
\cod{\mmc}{\begin{tabular}{c}$(n+1)m$\\$\downarrow$\end{tabular}}
\cod{2*\aaa+\bbc/2}{1+a+b}
\gbox{1}{-\repw-19/8*1.5}{-\repw-19/8*1.5+1.5}{0}
\gbox{1}{-\repw-19/8*1.5+1.5}{-\repw-19/8*1.5+3}{1}
\gbox{1}{-\repw-19/8*1.5+3}{-\repw}{2}
\node at (-\repw-0.1,1.6) [right] {$\times n$};
\gbox{1}{0}{\aaa-\bbc/2}{0}
\gbox{1}{\aaa-\bbc/2}{2*\aaa-\bbc}{2}
\gbox{1}{2*\aaa-\bbc}{\aaa}{3}
\gbox{1}{\aaa}{2*\aaa-\bbc/2}{1}
\gbox{1}{2*\aaa-\bbc/2}{\bbc}{0}
\gbox{1}{\bbc}{3*\aaa-\bbc}{3}
\gbox{1}{3*\aaa-\bbc}{\aaa+\bbc/2}{2}
\gbox{1}{\aaa+\bbc/2}{3*\aaa-\bbc/2}{0}
\gbox{1}{3*\aaa-\bbc/2}{\mmc}{1}
\gbox{1}{\mmc}{2*\aaa+\bbc/2}{0}
\end{tikzpicture} &&                                          (W_n^2)\\
&\begin{tikzpicture}[baseline=1.65cm]
\cod{-\repw-19/8*1.5}{0}
\cod{-\repw-19/8*1.5+1.5}{a}
\cod{-\repw-19/8*1.5+3}{2a}
\cod{-\repw}{m}
\cod{0}{nm}
\cod{\aaa-\bbc/2}{1}
\cod{2*\aaa-\bbc}{(n-1)m+3a}
\cod{\aaa}{nm+a}
\cod{2*\aaa-\bbc/2}{\begin{tabular}{c}$1+a$\\$\downarrow$\end{tabular}}
\cod{\bbc}{\quad\  nm+b}
\cod{\aaa+\bbc/2}{1+b}
\cod{3*\aaa-\bbc/2}{1+2a\qquad}
\cod{\mmc}{\begin{tabular}{c}$(n+1)m$\\$\downarrow$\end{tabular}}
\cod{2*\aaa+\bbc/2}{1+a+b}
\gbox{1}{-\repw-19/8*1.5}{-\repw-19/8*1.5+1.5}{0}
\gbox{1}{-\repw-19/8*1.5+1.5}{-\repw-19/8*1.5+3}{1}
\gbox{1}{-\repw-19/8*1.5+3}{-\repw}{2}
\node at (-\repw-0.1,1.6) [right] {$\times n$};
\gbox{1}{0}{\aaa-\bbc/2}{0}
\gbox{1}{\aaa-\bbc/2}{2*\aaa-\bbc}{2}
\gbox{1}{2*\aaa-\bbc}{\aaa}{3}
\gbox{1}{\aaa}{2*\aaa-\bbc/2}{1}
\gbox{1}{2*\aaa-\bbc/2}{\bbc}{0}
\gbox{1}{\bbc}{\aaa+\bbc/2}{2}
\gbox{1}{\aaa+\bbc/2}{3*\aaa-\bbc/2}{0}
\gbox{1}{3*\aaa-\bbc/2}{\mmc}{1}
\gbox{1}{\mmc}{2*\aaa+\bbc/2}{0}
\end{tikzpicture} &&                                          (W_n^3)\\
&\begin{tikzpicture}[baseline=1.65cm]
\cod{-\repw-19/8*1.5}{0}
\cod{-\repw-19/8*1.5+1.5}{a}
\cod{-\repw-19/8*1.5+3}{2a}
\cod{-\repw}{m}
\cod{0}{nm}
\cod{\aaa-\bbc/2}{1}
\cod{\aaa}{nm+a}
\cod{2*\aaa-\bbc/2}{\begin{tabular}{c}$1+a$\\$\downarrow$\end{tabular}}
\cod{\bbc}{\quad\  nm+b}
\cod{\aaa+\bbc/2}{1+b}
\cod{3*\aaa-\bbc/2}{1+2a\qquad}
\cod{\mmc}{\begin{tabular}{c}$(n+1)m$\\$\downarrow$\end{tabular}}
\cod{2*\aaa+\bbc/2}{1+a+b}
\gbox{1}{-\repw-19/8*1.5}{-\repw-19/8*1.5+1.5}{0}
\gbox{1}{-\repw-19/8*1.5+1.5}{-\repw-19/8*1.5+3}{1}
\gbox{1}{-\repw-19/8*1.5+3}{-\repw}{2}
\node at (-\repw-0.1,1.6) [right] {$\times n$};
\gbox{1}{0}{\aaa-\bbc/2}{0}
\gbox{1}{\aaa-\bbc/2}{\aaa}{3}
\gbox{1}{\aaa}{2*\aaa-\bbc/2}{1}
\gbox{1}{2*\aaa-\bbc/2}{\bbc}{0}
\gbox{1}{\bbc}{\aaa+\bbc/2}{2}
\gbox{1}{\aaa+\bbc/2}{3*\aaa-\bbc/2}{0}
\gbox{1}{3*\aaa-\bbc/2}{\mmc}{1}
\gbox{1}{\mmc}{2*\aaa+\bbc/2}{0}
\end{tikzpicture} &&                                          (W_n^4)
\end{alignedat}
\right. .
\end{equation*}

First we prove that our functions satisfy the condition of Proposition~\ref{prop:fg} when $x<1$, namely
\begin{equation*}
(\dag)\ f(x) =\mex\{f( x-s) \mid s\in S_{\leq x}\}\text{ for }0\leq x< 1.
\end{equation*}
Note that we put the numbers to the following diagrams as we mentioned in Definition-Notation~\ref{dntt:diag} (\ref{enum:per})(\ref{enum:num}). 
The diagram $\langle 0 \rangle$ below is for $x\in [0,m+a)$ and each diagram $\langle i \rangle$ is for $x\in [im+a,(i+1)m+a)$. 
\begin{align}
\setcounter{equation}{-1}
&\begin{tikzpicture}
\node at (-.2,0.72) [left] {$f(x-b)$};
\node at (-.2,1.22) [left] {$f(x-a)$};
\node at (-.2,1.72) [left] {$f(x)$};
\node at (-.2,2.22) [left] {$x\ $};
\cod{0}{0}
\cod{\aaa}{a}
\cod{\bbb}{b}
\cod{2*\aaa}{2a}
\cod{\aaa+\bbb}{m}
\cod{3*\aaa}{3a}
\cod{2*\aaa+\bbb}{m+a}
\gbox{1}{0}{\aaa}{0} \lper{1}{0}
\gbox{2}{\aaa}{2*\aaa}{0} \lper{2}{\aaa}
\gbox{3}{\bbb}{\aaa+\bbb}{0} \lper{3}{\bbb}
\gbox{1}{\aaa}{2*\aaa}{1}
\gbox{2}{2*\aaa}{3*\aaa}{1}
\gbox{3}{\aaa+\bbb}{2*\aaa+\bbb}{1}
\gbox{1}{2*\aaa}{\aaa+\bbb}{2}
\gbox{2}{3*\aaa}{2*\aaa+\bbb}{2}
\gbox{1}{\aaa+\bbb}{2*\aaa+\bbb}{0}
\end{tikzpicture}\\
&\begin{tikzpicture}
\node at (-.4,0.72) [left] {$f(x-b)$};
\node at (-.4,1.22) [left] {$f(x-a)$};
\node at (-.4,1.72) [left] {$f(x)$};
\node at (-.4,2.22) [left] {$x\ $};
\cod{0}{m+a}
\cod{\bbb-\aaa}{m+b}
\cod{\aaa}{m+2a}
\cod{\bbb}{2m}
\cod{2*\aaa}{m+3a}
\cod{\mmm}{2m+a}
\gbox{3}{0}{\bbb-\aaa}{2}
\gbox{2}{0}{\aaa}{0}
\gbox{3}{\bbb-\aaa}{\bbb}{0}
\gbox{1}{0}{\aaa}{1}
\gbox{2}{\aaa}{2*\aaa}{1}
\gbox{3}{\bbb}{\aaa+\bbb}{1}
\gbox{1}{\aaa}{\bbb}{2}
\gbox{2}{2*\aaa}{\aaa+\bbb}{2}
\gbox{1}{\bbb}{\aaa+\bbb}{0}
\end{tikzpicture}\\
&\begin{tikzpicture}
\node at (-.5,0.72) [left] {$f(x-b)$};
\node at (-.5,1.22) [left] {$f(x-a)$};
\node at (-.5,1.72) [left] {$f(x)$};
\node at (-.5,2.22) [left] {$x\ $};
\cod{0}{2m+a}
\cod{\bbb-\aaa}{2m+b}
\cod{\aaa}{2m+2a}
\cod{\bbb}{3m}
\cod{2*\aaa}{2m+3a\quad}
\cod{\mmm}{3m+a}
\gbox{3}{0}{\bbb-\aaa}{2}
\gbox{2}{0}{\aaa}{0}
\gbox{3}{\bbb-\aaa}{\bbb}{0}
\gbox{1}{0}{\aaa}{1}
\gbox{2}{\aaa}{2*\aaa}{1}
\gbox{3}{\bbb}{\aaa+\bbb}{1}
\gbox{1}{\aaa}{\bbb}{2}
\gbox{2}{2*\aaa}{\aaa+\bbb}{2}
\gbox{1}{\bbb}{\aaa+\bbb}{0}
\end{tikzpicture}\\
&\cdots\notag\\
&\begin{tikzpicture}
\node at (-.5,0.72) [left] {$f(x-b)$};
\node at (-.5,1.22) [left] {$f(x-a)$};
\node at (-.5,1.72) [left] {$f(x)$};
\node at (-.5,2.22) [left] {$x\ $};
\cod{0}{km+a}
\cod{\bbb-\aaa}{km+b}
\cod{\aaa}{km+2a\quad}
\cod{\bbb}{\quad(k+1)m}
\cod{2*\aaa}{km+3a\qquad}
\cod{\mmm}{\qquad(k+1)m+a}
\gbox{3}{0}{\bbb-\aaa}{2}
\gbox{2}{0}{\aaa}{0}
\gbox{3}{\bbb-\aaa}{\bbb}{0}
\gbox{1}{0}{\aaa}{1}
\gbox{2}{\aaa}{2*\aaa}{1}
\gbox{3}{\bbb}{\aaa+\bbb}{1}
\gbox{1}{\aaa}{\bbb}{2}
\gbox{2}{2*\aaa}{\aaa+\bbb}{2}
\gbox{1}{\bbb}{\aaa+\bbb}{0}
\end{tikzpicture}\tag{$k$}\\
&\cdots.\notag
\end{align}
This pattern with the local period $m=a+b$ (call this pattern as the local pattern) continues for $ k=0,1,2,\ldots $ until $ x$ reaches $ 1$, which concludes the proof for ($\dag$). As $f|_{[0,1)}=\varphi|_{[0,1)}$, the same holds for $\varphi$.

For the rest of proof, we consider when $ x$ reaches 1, i.e., the new row $ f( x-1)$ occurs. Using the diagram numbers, $x=1$ is in the diagram $\langle n^{*}\rangle$ where $ n^{*} \coloneqq \displaystyle\left\lfloor \frac{1-a}{m}\right\rfloor =\begin{cases}
n & (G_n\text{ or }R_n)\\
n-1 & (B_n\text{ or }W_n)
\end{cases}$, with the length $ \lambda $ from $ x=n^* m+a$ (the left end of diagram $\langle n^*\rangle$ except for when $n^*=0$).
\begin{equation}
\begin{tikzpicture}
\node at (-.6,0.22) [left] {$f(x-1)$};
\node at (-.6,0.72) [left] {$f(x-b)$};
\node at (-.6,1.22) [left] {$f(x-a)$};
\node at (-.6,1.72) [left] {$f(x)$};
\node at (-.6,2.22) [left] {$x\ $};
\cod{0}{n^*m+a}
\cod{2}{1}
\node at (0,2.25-.5*1) [right] {$1$};
\node at (0,2.25-.5*2) [right] {$0$};
\node at (0,2.25-.5*3) [right] {$2$};
\draw (2,2) --(0,2) --(0,.5) --(2,.5) (0,1.5) --(2,1.5) (0,1) --(2,1);
\draw [dashed] (2,2)  [rounded corners=6pt] -- (9,2) --(9,1.5) --(2,1.5) (9-.16,1.5)--(9,1.5) --(9,1) --(2,1) (9-.16,1)--(9,1) --(9,.5) --(2+\aaa,.5) (9-.16,.5)--(9,.5) --(9,0) --(2+\aaa,0);
\gbox{4}{2}{2+\aaa}{0} \lper{4}{2}
\draw (0,.25) [|<->] --(2,.25);
\node at (1,.25) [below] {$\lambda$};
\end{tikzpicture}\tag{$n^*$}
\end{equation}
We will check, for $G_n$ and $R_n^1$ as a representative,
\begin{itemize}
\item how the local pattern collapses (though it does not for $G_n$),
\item how the local pattern reappears for $x\geq1+a$ for $R_n$ and $x\geq1+b$ for $B_n$, and 
\item that the new row $f(x-1)$ is compatible with the pure period for all $x\geq1$,
\end{itemize}
so that $f(x)=\mex\{f(x-s)\mid s\in S_{\leq x}\}$ holds for $x\geq0$; and we check for $W_n^1$ that $\varphi=g|_{[0,1+a+b)}$ so that the Nim value function does not have the pure period $a+b$, $1+a$ nor $1+b$. For other regions, our check is similar to them. See Appendix~\ref{appb}.

Fig.~\ref{fig:lambda} indicates the correspondence between $\lambda$ and the region, depending on whether $a\leq b<\displaystyle\frac{4}{3}a$, $\displaystyle\frac{4}{3}a\leq b<\frac{3}{2}a$ or $\displaystyle\frac{3}{2}a\leq b<2a$.
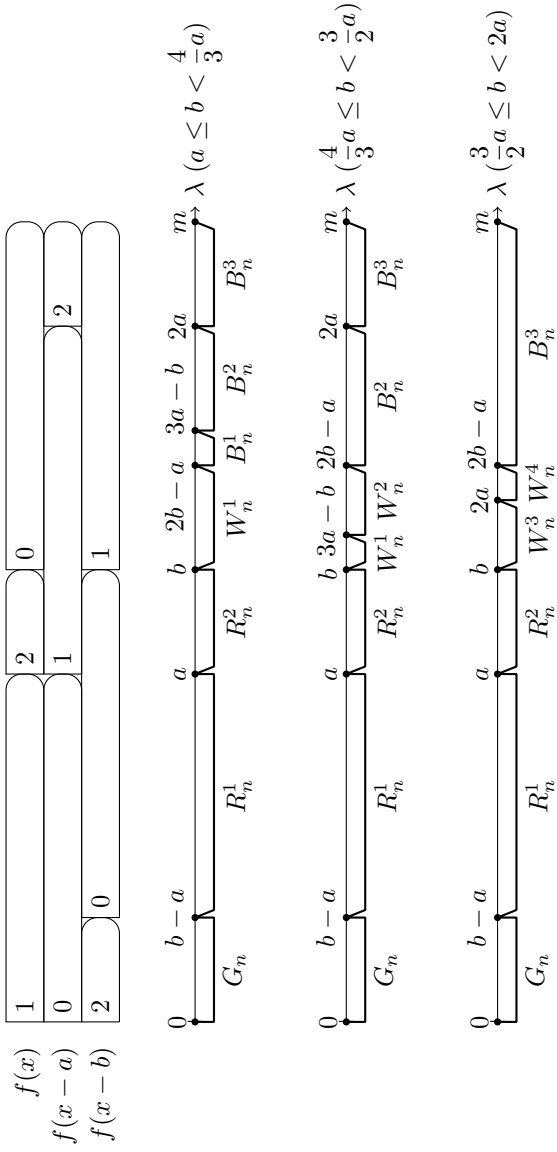
\begin{figure}[H]  % Used: TikZ
\begin{center}
\begin{tikzpicture}
\node at (-.2,3.22) [left] {$f(x-b)$};
\node at (-.2,3.72) [left] {$f(x-a)$};
\node at (-.2,4.22) [left] {$f(x)$};
\cod{0}{0}
\cod{\bbb-\aaa}{b-a}
\cod{\aaa}{a}
\cod{\bbb}{b}
\cod{2*\bbb-\aaa}{2b-a\qquad\ }
\cod{3*\aaa-\bbb}{\qquad\  3a-b}
\cod{2*\aaa}{2a}
\cod{\mmm}{m}
\codd{0}{0}
\codd{\bbb-\aaa}{b-a}
\codd{\aaa}{a}
\codd{\bbb}{b\ }
\codd{4/3*\bbb-\aaa/3}{\quad 3a-b}
\codd{2*\bbb-\aaa}{\qquad\ 2b-a}
\codd{2*\aaa}{2a}
\codd{\mmm}{m}
\coddd{0}{0}
\coddd{\bbb-\aaa}{b-a}
\coddd{\aaa}{a}
\coddd{\bbb}{b}
\coddd{5/3*\bbb-2/3*\aaa}{2a}
\coddd{2*\bbb-\aaa}{\qquad\ 2b-a}
\coddd{\mmm}{m}
\gbox{(-2)}{0}{\bbb-\aaa}{2}
\gbox{(-3)}{0}{\aaa}{0}
\gbox{(-2)}{\bbb-\aaa}{\bbb}{0}
\gbox{(-4)}{0}{\aaa}{1}
\gbox{(-3)}{\aaa}{2*\aaa}{1}
\gbox{(-2)}{\bbb}{\aaa+\bbb}{1}
\gbox{(-4)}{\aaa}{\bbb}{2}
\gbox{(-3)}{2*\aaa}{\aaa+\bbb}{2}
\gbox{(-4)}{\bbb}{\aaa+\bbb}{0}
\draw (0,2) [|->] --(\mmm+.2,2);
\node at (\mmm+.2,2) [right] {$\lambda\ 
(a\leq b<\displaystyle\frac{4}{3}a)$};
\draw (0,0) [|->] --(\mmm+.2,0);
\node at (\mmm+.2,0) [right] {$\lambda\ 
(\displaystyle\frac{4}{3}a\leq b<\frac{3}{2}a)$};
\draw (0,-2) [|->] --(\mmm+.2,-2);
\node at (\mmm+.2,-2) [right] {$\lambda\ 
(\displaystyle\frac{3}{2}a\leq b<2a)$};
\draw [thick, line join=bevel] (0,2) --(0,2-.25) --(\bbb-\aaa,2-.25) --(\bbb-\aaa,2)
--(\bbb-\aaa+.1,2-.25) --(\aaa,2-.25) --(\aaa,2)
--(\aaa+.1,2-.25) --(\bbb,2-.25) --(\bbb,2)
--(\bbb+.1,2-.25) --(2*\bbb-\aaa-.1,2-.25)
--(2*\bbb-\aaa,2) --(2*\bbb-\aaa,2-.25) --(3*\aaa-\bbb-.1,2-.25) --(3*\aaa-\bbb,2)
--(3*\aaa-\bbb,2-.25) --(2*\aaa-.1,2-.25) --(2*\aaa,2)
--(2*\aaa,2-.25) --(\mmm-.1,2-.25) --(\mmm,2);
\node at (\bbb/2-\aaa/2,2-.25) [below] {$G_n$};
\node at (\bbb/2,2-.25) [below] {$R_n^1$};
\node at (\aaa/2+\bbb/2,2-.25) [below] {$R_n^2$};
\node at (3/2*\bbb-\aaa/2,2-.25) [below] {$W_n^1$};
\node at (\aaa+\bbb/2,2-.25) [below] {$B_n^1$};
\node at (5/2*\aaa-\bbb/2,2-.25) [below] {$B_n^2$};
\node at (3/2*\aaa+\bbb/2,2-.25) [below] {$B_n^3$};
\draw [thick, line join=bevel] (0,0) --(0,-.25) --(\bbb-\aaa,-.25) --(\bbb-\aaa,0)
--(\bbb-\aaa+.1,-.25) --(\aaa,-.25) --(\aaa,0)
--(\aaa+.1,-.25) --(\bbb,-.25) --(\bbb,0)
--(\bbb+.1,-.25) 
--(4/3*\bbb-\aaa/3-.1,-.25) --(4/3*\bbb-\aaa/3,0) --(4/3*\bbb-\aaa/3,-.25)
--(2*\bbb-\aaa-.1,-.25)
--(2*\bbb-\aaa,0) --(2*\bbb-\aaa,-.25) --(2*\aaa-.1,-.25) --(2*\aaa,0)
--(2*\aaa,-.25) --(\mmm-.1,-.25) --(\mmm,0);
\node at (\bbb/2-\aaa/2,-.25) [below] {$G_n$};
\node at (\bbb/2,-.25) [below] {$R_n^1$};
\node at (\aaa/2+\bbb/2,-.25) [below] {$R_n^2$};
\node at (7/6*\bbb-\aaa/6,-.25) [below] {$W_n^1$};
\node at (5/3*\bbb-2*\aaa/3,-.25) [below] {$W_n^2$};
\node at (\bbb+\aaa/2,-.25) [below] {$B_n^2$};
\node at (3/2*\aaa+\bbb/2,-.25) [below] {$B_n^3$};
\draw [thick, line join=bevel] (0,-2) --(0,-2.25) --(\bbb-\aaa,-2.25) --(\bbb-\aaa,-2)
--(\bbb-\aaa+.1,-2.25) --(\aaa,-2.25) --(\aaa,-2)
--(\aaa+.1,-2.25) --(\bbb,-2.25) --(\bbb,-2)
--(\bbb+.1,-2.25) 
--(5/3*\bbb-2/3*\aaa-.1,-2.25) --(5/3*\bbb-2/3*\aaa,-2) --(5/3*\bbb-2/3*\aaa, -2.25)
--(2*\bbb-\aaa-.1,-2.25)
--(2*\bbb-\aaa,-2) --(2*\bbb-\aaa,-2.25) --(\mmm-.1,-2.25) --(\mmm,-2);
\node at (\bbb/2-\aaa/2,-2.25) [below] {$G_n$};
\node at (\bbb/2,-2.25) [below] {$R_n^1$};
\node at (\aaa/2+\bbb/2,-2.25) [below] {$R_n^2$};
\node at (4/3*\bbb-\aaa/3,-2.25) [below] {$W_n^3$};
\node at (11/6*\bbb-5/6*\aaa,-2.25) [below] {$W_n^4$};
\node at (3/2*\bbb,-2.25) [below] {$B_n^3$};
\end{tikzpicture}
\caption{Relation between $\lambda$ and the pure period.}\label{fig:lambda}
\end{center}
\end{figure}
Comparing the diagrams for each case, the meaning of the borders of the parameter $\lambda$ can be found.

\newpage
\begin{enumerate}
\item For $G_n$ where $ 0\leq \lambda=1-a-nm \leq b-a$,
\begin{align}
&\begin{tikzpicture}
\node at (-.2,0.22) [left] {$f(x-1)$};
\node at (-.2,0.72) [left] {$f(x-b)$};
\node at (-.2,1.22) [left] {$f(x-a)$};
\node at (-.2,1.72) [left] {$f(x)$};
\node at (-.2,2.22) [left] {$x\ $};
\cod{0}{\begin{tabular}{c}$nm+a$\\$\downarrow$\end{tabular}}
\cod{\bbb/3-\aaa/3}{1}
\cod{\bbb-\aaa}{\begin{tabular}{c}$nm+b$\\$\downarrow$\end{tabular}}
\cod{\aaa}{\begin{tabular}{c}$nm+2a\quad$\\$\downarrow$\end{tabular}}
\cod{\bbb/3+2/3*\aaa}{\quad 1+a}
\cod{\bbb}{\begin{tabular}{c}$\quad (n+1)m$\\$\downarrow$\end{tabular}}
\cod{2*\aaa}{\begin{tabular}{c}$nm+3a\qquad$\\$\downarrow$\end{tabular}}
\cod{\bbb/3+5/3*\aaa}{\quad 1+2a}
\cod{\mmm}{\begin{tabular}{c}$\qquad(n+1)m+a$\\$\downarrow$\end{tabular}}
\gbox{1}{0}{\aaa}{1}
\gbox{1}{\aaa}{\bbb}{2} \rper{1}{\bbb}
\gbox{1}{\bbb}{\mmm}{0} \lper{1}{\bbb}
\gbox{2}{0}{\aaa}{0} \lper{2}{0}
\gbox{2}{\aaa}{2*\aaa}{1}
\gbox{2}{2*\aaa}{\mmm}{2} \rper{2}{\mmm}
\gbox{3}{0}{\bbb-\aaa}{2} \rper{3}{\bbb-\aaa}
\gbox{3}{\bbb-\aaa}{\bbb}{0} \lper{3}{\bbb-\aaa}
\gbox{3}{\bbb}{\mmm}{1}
\gbox{4}{\bbb/3-\aaa/3}{\bbb/3+2/3*\aaa}{0} \lper{4}{\bbb/3-\aaa/3}
\gbox{4}{\bbb/3+2/3*\aaa}{\bbb/3+5/3*\aaa}{1}
\gbox{4}{\bbb/3+5/3*\aaa}{4/3*\bbb+2/3*\aaa}{2} \rper{4}{4/3*\bbb+2/3*\aaa}
\end{tikzpicture}\tag{$n$}\\
&\begin{tikzpicture}
\node at (-1.1,0.22) [left] {$f(x-1)$};
\node at (-1.1,0.72) [left] {$f(x-b)$};
\node at (-1.1,1.22) [left] {$f(x-a)$};
\node at (-1.1,1.72) [left] {$f(x)$};
\node at (-1.1,2.22) [left] {$x\ $};
\cod{0}{\begin{tabular}{c}$(n+1)m+a\qquad$\\$\downarrow$\end{tabular}}
\cod{\bbb/3-\aaa/3}{\quad 1+b}
\cod{\bbb-\aaa}{\begin{tabular}{c}$\qquad\quad(n+1)m+b$\\$\downarrow$\end{tabular}}
\cod{\aaa}{(n+1)m+2a\qquad\qquad\quad}
\cod{2/3*\aaa+\bbb/3}{\begin{tabular}{c}$1+m$\\$\downarrow$\end{tabular}}
\cod{\bbb}{\quad (n+2)m}
\cod{2*\aaa}{(n+1)m+3a\qquad\qquad\quad}
\cod{\bbb/3+5/3*\aaa}{\begin{tabular}{c}$1+m+a$\\$\downarrow$\end{tabular}}
\cod{\mmm}{\qquad\quad(n+2)m+a}
\gbox{1}{0}{\aaa}{1}
\gbox{1}{\aaa}{\bbb}{2} \rper{1}{\bbb}
\gbox{1}{\bbb}{\mmm}{0} \lper{1}{\bbb}
\gbox{2}{0}{\aaa}{0} \lper{2}{0}
\gbox{2}{\aaa}{2*\aaa}{1}
\gbox{2}{2*\aaa}{\mmm}{2} \rper{2}{\mmm}
\gbox{3}{0}{\bbb-\aaa}{2} \rper{3}{\bbb-\aaa}
\gbox{3}{\bbb-\aaa}{\bbb}{0} \lper{3}{\bbb-\aaa}
\gbox{3}{\bbb}{\mmm}{1}
\gbox{4}{-2/3*\bbb+2/3*\aaa}{\bbb/3-\aaa/3}{2} \rper{4}{\bbb/3-\aaa/3}
\gbox{4}{\bbb/3-\aaa/3}{\bbb/3+2/3*\aaa}{0} \lper{4}{\bbb/3-\aaa/3}
\gbox{4}{\bbb/3+2/3*\aaa}{\bbb/3+5/3*\aaa}{1}
\gbox{4}{\bbb/3+5/3*\aaa}{4/3*\bbb+2/3*\aaa}{2} \rper{4}{4/3*\bbb+2/3*\aaa}
\end{tikzpicture}\tag{$n+1$}
\end{align}
As every value in the row $ f( x-1)$ also appears at the same time in either of rows $ f( x-a)$ or $ f( x-b)$, we get
\begin{equation*}
\{f( x-s) \mid s\in S\} =\{f( x-a) ,f( x-b)\}
\end{equation*}
for $ x\geq 1$, implying
\begin{equation*}
f(x) =\mex\{f( x-s) \mid s\in S\} \ ( x\geq 1).
\end{equation*}

\newpage
\item For $R_n^1$ where $ b-a< \lambda=1-a-nm \leq a$,
\begin{align}
&\begin{tikzpicture}
\node at (-.5,0.22) [left] {$f(x-1)$};
\node at (-.5,0.72) [left] {$f(x-b)$};
\node at (-.5,1.22) [left] {$f(x-a)$};
\node at (-.5,1.72) [left] {$f(x)$};
\node at (-.5,2.22) [left] {$x\ $};
\cod{0}{nm+a}
\cod{\bbb-\aaa}{nm+b}
\cod{\aaa/2}{1}
\cod{\aaa}{nm+2a}
\cod{\bbb}{\begin{tabular}{c}$(n+1)m$\\$\downarrow$\end{tabular}}
\cod{3/2*\aaa}{1+a}
\cod{2*\aaa}{nm+3a}
\cod{\mmm}{\begin{tabular}{c}$(n+1)m+a$\\$\downarrow$\end{tabular}}
\cod{5/2*\aaa}{1+2a}
\gbox{1}{0}{\aaa}{1}
\gbox{1}{\aaa}{3/2*\aaa}{2} \rper{1}{3/2*\aaa}
\gbox{1}{3/2*\aaa}{5/2*\aaa}{0} \lper{1}{3/2*\aaa}
\gbox{2}{0}{\aaa}{0}
\gbox{2}{\aaa}{2*\aaa}{1}
\gbox{2}{2*\aaa}{5/2*\aaa}{2} \rper{2}{5/2*\aaa}
\gbox{3}{0}{\bbb-\aaa}{2}
\gbox{3}{\bbb-\aaa}{\bbb}{0}
\gbox{3}{\bbb}{\mmm}{1}
\gbox{3}{\mmm}{\bbb+3/2*\aaa}{2} \rper{3}{\bbb+3/2*\aaa}
\gbox{4}{\aaa/2}{3/2*\aaa}{0} \lper{4}{\aaa/2}
\gbox{4}{3/2*\aaa}{5/2*\aaa}{1}
\end{tikzpicture}\tag{$n$}
\end{align}
Notice that due to $f(x-1)=0$ for $(n+1) m\leq x< 1+a$, the box $f(x)=2$ is stretched from length $b-a$ to $1-nm-a$. It trims the next box $f(x)=0$, but the stretched box $f(x-a)=2$ stretches it and the length is $a$ after all.
\begin{align}
&\begin{tikzpicture}
\node at (\bbb-3/2*\aaa-.2,0.22) [left] {$f(x-1)$};
\node at (\bbb-3/2*\aaa-.2,0.72) [left] {$f(x-b)$};
\node at (\bbb-3/2*\aaa-.2,1.22) [left] {$f(x-a)$};
\node at (\bbb-3/2*\aaa-.2,1.72) [left] {$f(x)$};
\node at (\bbb-3/2*\aaa-.2,2.22) [left] {$x\ $};
\cod{0}{1+2a}
\cod{\bbb-\aaa}{\quad 1+a+b}
\cod{\aaa}{1+3a}
\cod{\bbb}{\quad 1+2a+b}
\cod{2*\aaa}{1+4a}
\cod{\mmm}{\quad 1+3a+b}
\gbox{1}{0}{\aaa}{1}
\gbox{1}{\aaa}{\bbb}{2}
\gbox{1}{\bbb}{\mmm}{0}
\gbox{2}{0}{\aaa}{0} \lper{2}{0}
\gbox{2}{\aaa}{2*\aaa}{1}
\gbox{2}{2*\aaa}{\mmm}{2}
\gbox{3}{\bbb-3/2*\aaa}{\bbb-\aaa}{2} \rper{3}{\bbb-\aaa}
\gbox{3}{\bbb-\aaa}{\bbb}{0} \lper{3}{\bbb-\aaa}
\gbox{3}{\bbb}{\mmm}{1}
\gbox{4}{0}{\bbb-\aaa}{2}
\gbox{4}{\bbb-\aaa}{\bbb}{0}
\gbox{4}{\bbb}{\mmm}{1}
\end{tikzpicture}\tag{$n+1$}
\end{align}
As you can see, for $1+a+b\leq x<1+nm+2a$ we have
\begin{align*}
f(x-1)&=f(x-(a+b)-1) &&(\text{the local period }a+b)\\
      &=f(x-(1+a)-b)\\
      &=f(x-b).      &&(\text{the period }1+a)
\end{align*}
Hence the diagrams $\langle n+1 \rangle$ to $\langle 2n \rangle$ have the local period $a+b$, essentially the same as the diagrams $\langle 0 \rangle$ to $\langle n-1 \rangle$.

\begin{align}
&\begin{tikzpicture}
\node at (-.2,0.22) [left] {$f(x-1)$};
\node at (-.2,0.72) [left] {$f(x-b)$};
\node at (-.2,1.22) [left] {$f(x-a)$};
\node at (-.2,1.72) [left] {$f(x)$};
\node at (-.2,2.22) [left] {$x\ $};
\cod{0}{\begin{tabular}{c}$1+nm+2a\qquad\quad$\\$\downarrow$\end{tabular}}
\cod{\bbb-\aaa}{\begin{tabular}{c}$\qquad\quad1+(n+1)m$\\$\downarrow$\end{tabular}}
\cod{\aaa/2}{2+a}
\cod{\aaa}{1+nm+3a}
\gbox{1}{0}{\aaa}{1}
\gbox{2}{0}{\aaa}{0}
\gbox{3}{0}{\bbb-\aaa}{2}
\gbox{3}{\bbb-\aaa}{\bbb}{0}
\gbox{4}{0}{\aaa/2}{2} \rper{4}{\aaa/2}
\gbox{4}{\aaa/2}{3/2*\aaa}{0} \lper{4}{\aaa/2}
\end{tikzpicture}\tag{$2n+1$}
\end{align}
Thus $ f(x) =\mex\{f( x-s) \mid s\in S\}$ for $ 1\leq x< 2+a$. Notice that, as $f$ is purely periodic with the period $1+a$ for $R_n^1$, for $x\geq2+a$, the diagram is exactly the same as the diagram $\langle n \rangle$, hence we have $ f(x) =\mex\{f( x-s) \mid s\in S\}$ for all $x\geq1$. We are done for $R_n^1$.\\

\item For $W_n^1$ where $b<\lambda=1-a-(n-1)m<\min(2b-a,3a-b)$,
\begin{align*}
&\begin{tikzpicture}
\node at (\aaa-\bbb-.2,0.22) [left] {$\varphi(x-1)$};
\node at (\aaa-\bbb-.2,0.72) [left] {$\varphi(x-b)$};
\node at (\aaa-\bbb-.2,1.22) [left] {$\varphi(x-a)$};
\node at (\aaa-\bbb-.2,1.72) [left] {$\varphi(x)$};
\node at (\aaa-\bbb-.2,2.22) [left] {$x\ $};
\cod{3/2*\aaa-\bbb}{1}
\cod{2*\aaa-\bbb}{(n-1)m+3a\quad}
\cod{\aaa}{\begin{tabular}{c}$nm+a$\\$\downarrow$\end{tabular}}
\cod{5/2*\aaa-\bbb}{1+a\quad}
\cod{\bbb}{\begin{tabular}{c}$nm+b$\\$\downarrow$\end{tabular}}
\cod{3/2*\aaa}{1+b}
\cod{3*\aaa-\bbb}{\begin{tabular}{c}\quad$(n-1)m+4a$\\$\downarrow$\end{tabular}}
\cod{2*\aaa}{nm+2a\qquad\quad}
\cod{7/2*\aaa-\bbb}{1+2a\quad}
\cod{\mmm}{\begin{tabular}{c}\quad$(n+1)m$\\$\downarrow$\end{tabular}}
\cod{5/2*\aaa}{\quad 1+a+b}
\gbox{1}{3/2*\aaa-\bbb}{2*\aaa-\bbb}{2}
\gbox{1}{2*\aaa-\bbb}{\aaa}{3}
\gbox{1}{\aaa}{5/2*\aaa-\bbb}{1}
\gbox{1}{5/2*\aaa-\bbb}{\bbb}{0}
\gbox{1}{\bbb}{3/2*\aaa}{3}
\gbox{1}{3/2*\aaa}{7/2*\aaa-\bbb}{0}
\gbox{1}{7/2*\aaa-\bbb}{\mmm}{1}
\gbox{1}{\mmm}{5/2*\aaa}{0}
\gbox{2}{\aaa-\bbb}{2*\aaa-\bbb}{1}
\gbox{2}{2*\aaa-\bbb}{\aaa}{2}
\gbox{2}{\aaa}{5/2*\aaa-\bbb}{0}
\gbox{2}{5/2*\aaa-\bbb}{3*\aaa-\bbb}{2}
\gbox{2}{3*\aaa-\bbb}{2*\aaa}{3}
\gbox{2}{2*\aaa}{7/2*\aaa-\bbb}{1}
\gbox{2}{7/2*\aaa-\bbb}{\mmm}{0}
\gbox{2}{\mmm}{5/2*\aaa}{3}
\gbox{3}{0}{\aaa}{1}
\gbox{3}{\aaa}{\bbb}{2}
\gbox{3}{\bbb}{3/2*\aaa}{0}
\gbox{3}{3/2*\aaa}{2*\aaa}{2}
\gbox{3}{2*\aaa}{\mmm}{3}
\gbox{3}{\mmm}{5/2*\aaa}{1}
\gbox{4}{3/2*\aaa-\bbb}{5/2*\aaa-\bbb}{0}
\gbox{4}{5/2*\aaa-\bbb}{7/2*\aaa-\bbb}{1}
\gbox{4}{7/2*\aaa-\bbb}{5/2*\aaa}{2}
\end{tikzpicture}
\end{align*}

\begin{itemize}
\item There is a box $\varphi(x)=3$, so $a+b$ cannot be a pure period.
\item $\varphi(x)$ is not constant $0$ for $1+a\leq x<1+2a$ or $1+b\leq x<1+a+b$, so $1+a$ and $1+b$ cannot be a pure period either.
\end{itemize}

\end{enumerate}

\end{proof}
\end{landscape}

\begin{corollary}\label{cor:a1}

For $ 0< a< 1$, let $n=\displaystyle\left\lceil\frac{1}{2a}\right\rceil-1$, then we have $0<\displaystyle\frac{1}{a}-2n\leq2$, and depending on its value, the Nim value function of $ G(S)$ for $ S=\{a,1\}$ is
\begin{equation*}
\begin{cases}
\begin{tikzpicture}[baseline=1.65cm]
\cod{0}{0} \cod{1.5}{a} \cod{3}{2a}
\cod{3+\repw}{2na} \cod{3+\repw+\aaa/3}{1}
\cod{3+\repw+\aaa}{(2n+1)a} \cod{3+\repw+4/3*\aaa}{1+a}
\node at (3-.1,1.6) [right] {$\times n$};
\gbox{1}{0}{1.5}{0} \lper{1}{0}
\gbox{1}{1.5}{3}{1} 
\gbox{1}{3+\repw}{3+\repw+\aaa/3}{0}
\gbox{1}{3+\repw+\aaa/3}{3+\repw+\aaa}{2}
\gbox{1}{3+\repw+\aaa}{3+\repw+4/3*\aaa}{1} \rper{1}{3+\repw+4/3*\aaa}
\end{tikzpicture}&(\frac{1}{a}-2n<1)\\
\begin{tikzpicture}[baseline=1.65cm]
\cod{0}{0} \cod{1.5}{a} \cod{3}{2a}
\gbox{1}{0}{1.5}{0} \lper{1}{0}
\gbox{1}{1.5}{3}{1} \rper{1}{3}
\end{tikzpicture}&(\frac{1}{a}-2n=1)\\
\begin{tikzpicture}[baseline=1.65cm]
\cod{0}{0} \cod{1.5}{a} \cod{3}{2a}
\cod{3+2*\repw}{(2n+2)a} \cod{3+2*\repw+\aaa/3}{1+a}
\node at (3-.1,1.6) [right] {$\times (n+1)$};
\gbox{1}{0}{1.5}{0} \lper{1}{0}
\gbox{1}{1.5}{3}{1} 
\gbox{1}{3+2*\repw}{3+2*\repw+\aaa/3}{2} \rper{1}{3+2*\repw+\aaa/3}
\end{tikzpicture}&(\frac{1}{a}-2n>1)
\end{cases}
\end{equation*}
with the minimum period
\begin{equation*}
\begin{cases}
2a & (\frac{1}{a}-2n=1)\\
1+a & (\text{otherwise})
\end{cases}.
\end{equation*}

\begin{proof}
Straightforward from Theorem~\ref{th:necktie} by making $b=a$. To be more precise,
\begin{enumerate}
\item if $0<\frac{1}{a}-2n<1$, then we have $2na<1<(2n+1)a$, so the point $(a,a)$ is in $B_n^1$;
\item if $\frac{1}{a}-2n=1$, then we have $(2n+1)a=1$, so the point $(a,a)$ is in $G_n$;
\item if $1<\frac{1}{a}-2n\leq 2$, then we have $(2n+1)a<1\leq(2n+2)a$, so the point $(a,a)$ is in $R_n^1$.
\end{enumerate}

\end{proof}

\end{corollary}

Fig.~\ref{fig:rgb} is a calculation result, which we built the claim of Theorem~\ref{th:necktie} originally from.

Each point has a coordinate $ ( a,b) =\left(\displaystyle\frac{k\sqrt{2}}{\lceil 40\sqrt{2}\rceil } ,\frac{l\sqrt{3}}{\lceil 40\sqrt{3}\rceil }\right) \ ( k,l=1,2,\ldots ,40)$, and the color indicates that the Nim value function of $ G( \{a,b,1\})$ is
\begin{equation*}
\begin{cases}
\text{purely periodic with the period } a+b & (\text{green})\\
\text{purely periodic with the period } 1+\min( a,b) & (\text{red})\\
\text{purely periodic with the period } 1+\max( a,b) & (\text{blue})\\
\text{other than above} & (\text{gray})
\end{cases}.
\end{equation*}

We found that there are several blue points in the region $ \max( a,b) >2\min( a,b)$, but no red or green points. Rikuki Okada (Tohoku University) conjectured the condition for them, on which he will write a paper in the future.

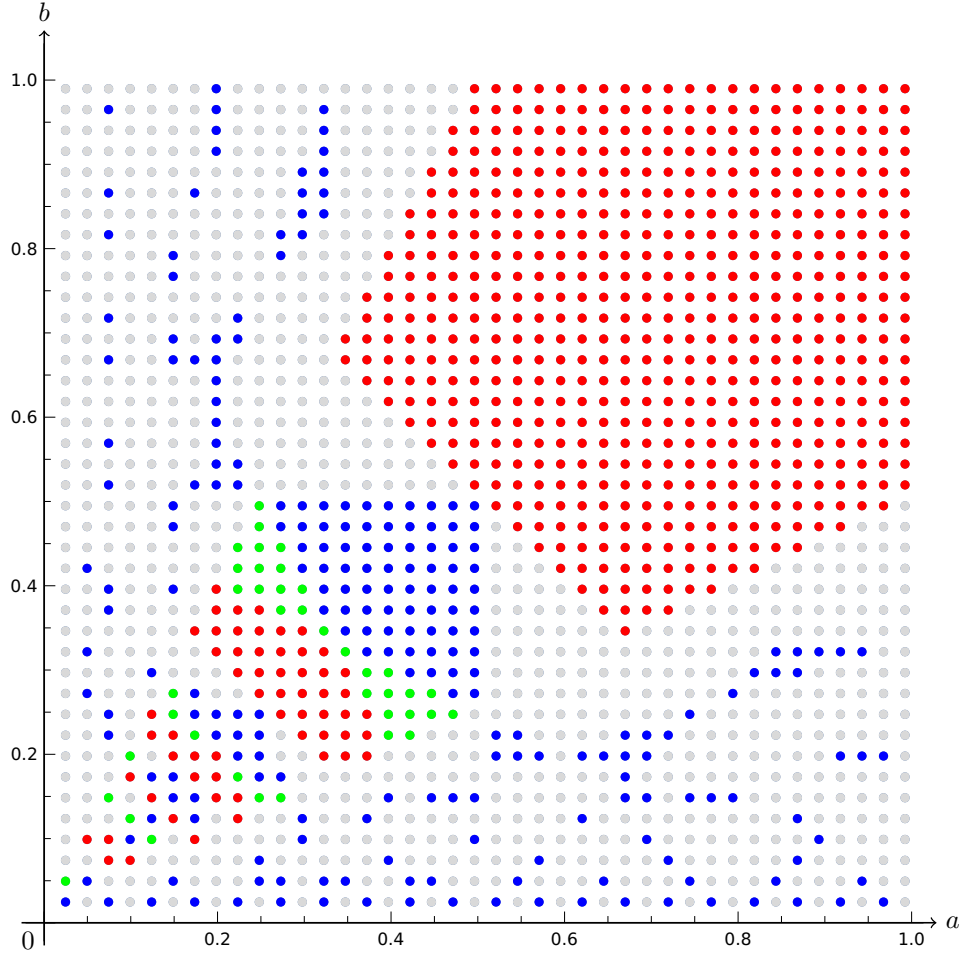
\begin{figure}[H] % Used: Mathematica and TikZ
\centering
\begin{tikzpicture}
\node at (5.634,5.605) {\includegraphics[width=0.93\linewidth]{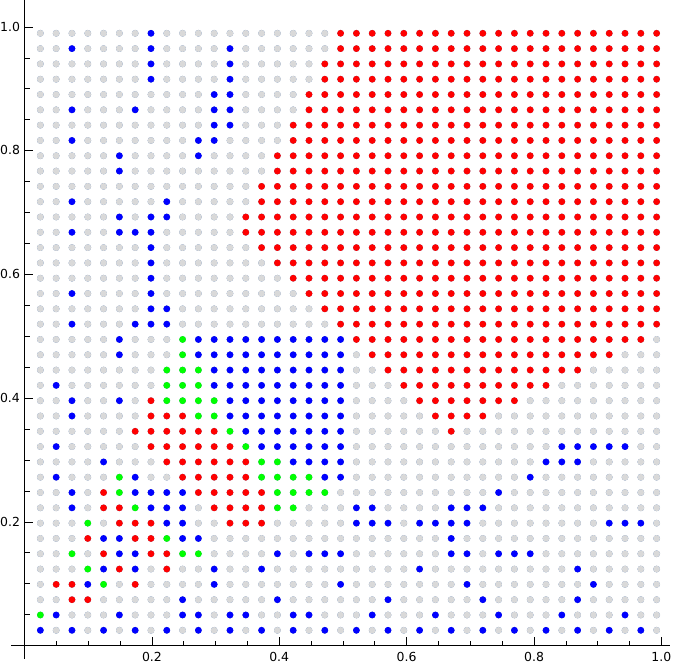}
};
\draw [line width=0.6pt] (0,-.3) [->] --(0,11.8);
\draw [line width=0.6pt] (-.3,0) [->] --(11.8,0);
\node at (0,0) [below left] {$0$};
\node at (0,11.8) [above] {$b$};
\node at (11.8,0) [right] {$a$};
\end{tikzpicture}
\caption{Experiment.}
\label{fig:rgb}
\end{figure}

Below is an example of such purely periodic Nim value function, located around $ ( 0.2,0.7)$ in Fig.~\ref{fig:rgb}.

\begin{ex}

Let $ S\coloneqq \{1,\sqrt{11} ,\sqrt{23}\} \ (\left(\frac{1}{\sqrt{23}},\frac{\sqrt{11}}{\sqrt{23}}\right) \approx ( 0.209,0.692))$. Then the Nim value function of $ G(S)$ is 
\begin{align*}
&\begin{tikzpicture}
\cod{0}{0} \cod{\ext}{1} \cod{2*\ext}{2} \cod{3*\ext}{3} \cod{4*\ext}{4}
\cod{\ext+\exa}{1+\sqrt{11}}
\gbox{1}{0}{\ext}{0} \lper{1}{0}
\gbox{1}{\ext}{2*\ext}{1}
\gbox{1}{2*\ext}{3*\ext}{0}
\gbox{1}{3*\ext}{4*\ext}{1}
\gbox{1}{4*\ext}{\ext+\exa}{2}
\end{tikzpicture}\\
&\begin{tikzpicture}[baseline=1.65cm]
\cod{\ext+\exa}{1+\sqrt{11}} \cod{\exb}{\sqrt{23}}
\cod{2*\ext+\exa}{2+\sqrt{11}}
\cod{\ext+\exb}{\begin{tabular}{c}$1+\sqrt{23}$\\$\downarrow$\end{tabular}}
\cod{3*\ext+\exa}{3+\sqrt{11}}
\cod{2*\ext+\exb}{\begin{tabular}{c}$2+\sqrt{23}$\\$\downarrow$\end{tabular}}
\cod{4*\ext+\exa}{4+\sqrt{11}}
\cod{3*\ext+\exb}{\begin{tabular}{c}$3+\sqrt{23}$\\$\downarrow$\end{tabular}} \cod{5*\ext}{5} \cod{6*\ext}{6} \cod{7*\ext}{7} \cod{8*\ext}{8\ }
\cod{\exa+\exb}{\qquad\qquad \sqrt{11}+\sqrt{23}}
\gbox{1}{\ext+\exa}{\exb}{0}
\gbox{1}{\exb}{5*\ext}{2}
\gbox{1}{5*\ext}{2*\ext+\exa}{3}
\gbox{1}{2*\ext+\exa}{\ext+\exb}{1}
\gbox{1}{\ext+\exb}{6*\ext}{3}
\gbox{1}{6*\ext}{3*\ext+\exa}{2}
\gbox{1}{3*\ext+\exa}{2*\ext+\exb}{0}
\gbox{1}{2*\ext+\exb}{7*\ext}{2}
\gbox{1}{7*\ext}{4*\ext+\exa}{3}
\gbox{1}{4*\ext+\exa}{3*\ext+\exb}{1}
\gbox{1}{3*\ext+\exb}{8*\ext}{3}
\gboxs{1}{8*\ext}{\exa+\exb}{2} \rper{1}{\exa+\exb}
\end{tikzpicture}.
\end{align*}

\end{ex}

\begin{conj}\label{conj:genper}

The Nim value function of $ G( S)$ for $ S=\{a,b,c\}$ with $0<a<b<c$ and $c\neq a+b$ (See Remark~\ref{rm:c=a+b} below for this condition) is periodic with a period $ a+b$, $ b+c$ or $ c+a$.

\end{conj}

\begin{rem}\label{rm:c=a+b}

For the case of $S=\{a,b,a+b\}$, we have shown that the period is $\lceil \dfrac{b}{a}\rceil a+b$ when $\lceil \dfrac{b}{a}\rceil$ is even, and $a(2b+r)$ when $\lceil \dfrac{b}{a}\rceil$ is odd. This result will be published in \citep{bhagat2026bracket}.

\end{rem}

\section{Mory sequence}

\begin{proposition}\label{prop:pnins}

If the Nim value function of $ G( S)$ is periodic with a period $ p$, then $ p\notin S$.

\end{proposition}

\begin{proof}

For $ x\geq \text{preperiod}$ and $s\in S$, we have $g(x+p)=g(x)$ and $g(x+s)\neq g(x)$. 

\end{proof}

\begin{definition}

For given $ G( S)$, we define the \uline{Mory sequence} $ \bm{p}_S=[ p_{i}]_{i}$ of $ S$ as follows:

$ p_{i}$ is the minimum period of the Nim value function of $ G( S\cup \{p_{0} ,p_{1} ,\ldots ,p_{i-1}\})$ if periodic.

\end{definition}

By Proposition~\ref{prop:pnins} we get $ p_{i} \notin \{p_{0} ,p_{1} ,\ldots ,p_{i-1}\}$ if defined. There are multiple results of Mory sequence for the classical Subtraction Nim in \citep{moriwaki2025mory}.

\begin{thm}

For $0<s_1<s_2<\cdots<s_{r-1}<s_r$, if $s_{j+1}\leq s_j+s_1\ (1\leq j<r)$, then the Mory sequence of $S=\{s_1,s_2,\ldots,s_r\}$ is
\begin{equation*}
\bm{p}_S=[s_1+s_r,\ 2s_1+s_r,\ 3s_1+s_r,\ldots]=[(i+1)s_1+s_r]_{i}.
\end{equation*}
Moreover, the Nim value function of $G(S_l)$ for $S_l=S\cup\{s_1+s_r ,2s_1+s_r,\ldots,ls_1+s_r\}$ is
\begin{equation*}
g_{S_l}(x)=f_l(x)\coloneqq\left\lfloor\frac{x\bmod\bigl((l+1)s_1+s_r\bigr)}{s_1}\right\rfloor\ (l=0,1,2,\ldots).
\end{equation*}

\end{thm}

\begin{proof}

Straightforward from Theorem~\ref{th:r001}, as each $\displaystyle\frac{S_l}{\max S_l}$ satisfies the assumption.
\end{proof}

\section{Future work}

\begin{enumerate}
\item The non-colored region in Fig.~\ref{fig:necktie}. We conjecture there is periodicity with the minimum preperiod $>0$ for the white region.
\item Let $ |S|= 4$ and calculate like Fig.~\ref{fig:rgb}.
\item Application 1: Ending-Partizan. It is the winning judge where the winner depends on the terminal position the last player makes, which was originally proposed by \citep{inoue2025ending}. For Continim $G(S)$, we can consider the game in which some $t\in[0,\min S)$ is given, two players $L$ and $R$ plays, and when the position $x$ gets less than $\min S$, the winner is the player $\begin{cases}
L &(x\in[0,t))\\
R &(x\in[t,\min S))
\end{cases}$.
In \citep{inoue2025ending}, dealing with the classic Subtraction Nim with $t=1$ and $\min S=2$, Inoue observed multiple games where the player $L$ can force a win from any sufficient large positions. Considering Continim, we may be able to find just the right border $t$ such that the game is fair to both of the two players $L$ and $R$.
\item Application 2: D.U.D.E.N.E.Y $D(S)$. It is the Subtraction Nim (on the normal play) where the player cannot choose the same removable number as the previous play. It is introduced in \citep{berlekamp2018winning} considering the case $S=\{1,2,\ldots,Y\}$ where $Y$ is a positive integer, and we got an idea of Continim while exploring the general case. Each position of D.U.D.E.N.E.Y is linked the information that it is initial one or obtained by subtracting certain $s\in S$. We are eager to find the Nim values of them for general $S$.
\end{enumerate}

\bibliography{conti}

\begin{appendices}

\section{Countably Infinite Continim}

For the classical Subtraction Nim, an infinite set is necessarily unbounded, while there are only finitely many choices from a position.
On the other hand, we can consider a countably infinite set that is bounded for the Continim. Here are some examples.

\begin{itemize}
\item $S\coloneqq\{1,\frac{3}{2},\frac{5}{3},\frac{7}{4},\ldots\}=\{2-\frac{1}{n}\mid n=1,2,\ldots\}$.
\begin{align*}
\begin{tikzpicture}
\node at (-.2,.22-1) [left] {$\vdots\quad$};
\node at (-.2,.22-.5) [left] {$g(x-\frac{7}{4})$};
\node at (-.2,.22) [left] {$g(x-\frac{5}{3})$};
\node at (-.2,.72) [left] {$g(x-\frac{3}{2})$};
\node at (-.2,1.22) [left] {$g(x-1)$};
\node at (-.2,1.72) [left] {$g(x)$};
\node at (-.2,2.22) [left] {$x\ $};
\cod{0}{0}
\cod{\epw}{1}
\cod{2*\epw}{2}
\cod{3*\epw}{3}
\cod{4*\epw}{4}
\cod{5*\epw}{5}
\cod{6*\epw}{6}
\cod{3/2*\epw}{\frac{3}{2}}
\cod{5/3*\epw}{\frac{5}{3}}
\cod{7/4*\epw}{\ \frac{7}{4}}
\cod{9/5*\epw}{}
\cod{11/6*\epw}{}
\cod{13/7*\epw}{}
\cod{15/8*\epw}{}
\gbox{1}{0}{\epw}{0} \lper{1}{0}
\gbox{1}{\epw}{2*\epw}{1}
\gbox{1}{2*\epw}{3*\epw}{2} \rper{1}{3*\epw}
\gbox{1}{3*\epw}{4*\epw}{0} \lper{1}{3*\epw}
\gbox{1}{4*\epw}{5*\epw}{1}
\gbox{1}{5*\epw}{6*\epw}{2} \rper{1}{6*\epw}
\gbox{2}{\epw}{2*\epw}{0} \lper{2}{\epw}
\gbox{2}{2*\epw}{3*\epw}{1}
\gbox{2}{3*\epw}{4*\epw}{2} \rper{2}{4*\epw}
\gbox{2}{4*\epw}{5*\epw}{0} \lper{2}{4*\epw}
\gbox{2}{5*\epw}{6*\epw}{1}
\gbox{3}{3/2*\epw}{5/2*\epw}{0} \lper{3}{3/2*\epw}
\gbox{3}{5/2*\epw}{7/2*\epw}{1}
\gbox{3}{7/2*\epw}{9/2*\epw}{2} \rper{3}{9/2*\epw}
\gbox{3}{9/2*\epw}{11/2*\epw}{0} \lper{3}{9/2*\epw}
\gbox{4}{5/3*\epw}{8/3*\epw}{0} \lper{4}{5/3*\epw}
\gbox{4}{8/3*\epw}{11/3*\epw}{1}
\gbox{4}{11/3*\epw}{14/3*\epw}{2} \rper{4}{14/3*\epw}
\gbox{4}{14/3*\epw}{17/3*\epw}{0} \lper{4}{14/3*\epw}
\gbox{5}{7/4*\epw}{11/4*\epw}{0} \lper{5}{7/4*\epw}
\gbox{5}{11/4*\epw}{15/4*\epw}{1}
\gbox{5}{15/4*\epw}{19/4*\epw}{2} \rper{5}{19/4*\epw}
\gbox{5}{19/4*\epw}{23/4*\epw}{0} \lper{5}{19/4*\epw}
\end{tikzpicture}
\end{align*}
We get $g(x)=\lfloor x\rfloor \bmod 3\ (x\geq 0)$, which is identical to the Nim value for $\{1,2\}$.
\item $S\coloneqq\{2,\frac{3}{2},\frac{4}{3},\frac{5}{4},\ldots\}=\{1+\frac{1}{n}\mid n=1,2,\ldots\}$. 
\begin{align*}
\begin{tikzpicture}
\node at (-.2,.22-1) [left] {$\vdots\quad$};
\node at (-.2,.22-.5) [left] {$g(x-\frac{7}{4})$};
\node at (-.2,.22) [left] {$g(x-\frac{5}{3})$};
\node at (-.2,.72) [left] {$g(x-\frac{3}{2})$};
\node at (-.2,1.22) [left] {$g(x-1)$};
\node at (-.2,1.72) [left] {$g(x)$};
\node at (-.2,2.22) [left] {$x\ $};
\cod{0}{0}
\cod{\epw}{1}
\cod{2*\epw}{2}
\cod{3*\epw}{3}
\cod{4*\epw}{4}
\cod{5*\epw}{5}
\cod{6*\epw}{6}
\cod{3/2*\epw}{\frac{3}{2}}
\cod{4/3*\epw}{\frac{4}{3}}
\cod{5/4*\epw}{\frac{5}{4}\ }
\cod{6/5*\epw}{}
\cod{7/6*\epw}{}
\cod{8/7*\epw}{}
\cod{9/8*\epw}{}
\gboxxx{1}{0}{\epw}{0} \lperr{1}{0}
\gboxx{1}{\epw}{2*\epw}{1}
\gboxx{1}{2*\epw}{3*\epw}{2} \rperr{1}{3*\epw}
\gboxx{1}{3*\epw}{4*\epw}{0} \lperr{1}{3*\epw}
\gboxx{1}{4*\epw}{5*\epw}{1}
\gboxx{1}{5*\epw}{6*\epw}{2} \rperr{1}{6*\epw}
\gboxxx{2}{2*\epw}{3*\epw}{0} \lperr{2}{2*\epw}
\gboxx{2}{3*\epw}{4*\epw}{1}
\gboxx{2}{4*\epw}{5*\epw}{2} \rperr{2}{5*\epw}
\gboxx{2}{5*\epw}{6*\epw}{0} \lperr{2}{5*\epw}
\gboxxx{3}{3/2*\epw}{5/2*\epw}{0} \lperr{3}{3/2*\epw}
\gboxx{3}{5/2*\epw}{7/2*\epw}{1}
\gboxx{3}{7/2*\epw}{9/2*\epw}{2} \rperr{3}{9/2*\epw}
\gboxx{3}{9/2*\epw}{11/2*\epw}{0} \lperr{3}{9/2*\epw}
\gboxxx{4}{4/3*\epw}{7/3*\epw}{0} \lperr{4}{4/3*\epw}
\gboxx{4}{7/3*\epw}{10/3*\epw}{1}
\gboxx{4}{10/3*\epw}{13/3*\epw}{2} \rperr{4}{13/3*\epw}
\gboxx{4}{13/3*\epw}{16/3*\epw}{0} \lperr{4}{13/3*\epw}
\gboxxx{5}{5/4*\epw}{9/4*\epw}{0} \lperr{5}{5/4*\epw}
\gboxx{5}{9/4*\epw}{13/4*\epw}{1}
\gboxx{5}{13/4*\epw}{17/4*\epw}{2} \rperr{5}{17/4*\epw}
\gboxx{5}{17/4*\epw}{21/4*\epw}{0} \lperr{5}{17/4*\epw}
\end{tikzpicture}
\end{align*}
The position $1$ is terminal and the position that is slightly greater than $1$ is playable once. In fact, we get 
\begin{equation*}
g(0)=0\text{ and }g(x)=(\lceil x\rceil-1)\bmod 3\ (x>0).
\end{equation*}
In this case, the Nim value function has the periodicity that does not acculately meet Definition~\ref{df:period}, namely
\begin{equation*}
g(x+p)=g(x)\ (x>x_0)\text{ but }g(x_0+p)\neq g(x_0)
\end{equation*}
for $x_0=0$ and $p=3$.
\item $S\coloneqq\{\frac{1}{n}\mid n=1,2,\ldots\}$. As $\inf S=0$, the game can be arbitrarily long. We can get immediately that all of the positions except $0$ and the removable numbers themselves are the Draw positions.
\end{itemize}

\begin{landscape}

\section{Proof of Theorem~\ref{th:necktie} and Remarks} \label{appb}

\begin{itemize}
\item For $R_n^2$ where $ a< \lambda=1-a-nm \leq b$,
\begin{align}
&\begin{tikzpicture}
\node at (-.6,0.22) [left] {$f(x-1)$};
\node at (-.6,0.72) [left] {$f(x-b)$};
\node at (-.6,1.22) [left] {$f(x-a)$};
\node at (-.6,1.72) [left] {$f(x)$};
\node at (-.6,2.22) [left] {$x\ $};
\cod{0}{nm+a}
\cod{\bbb-\aaa}{nm+b}
\cod{\aaa}{nm+2a\qquad}
\cod{\aaa/2+\bbb/2}{1}
\cod{\bbb}{\qquad (n+1)m}
\cod{2*\aaa}{nm+3a\qquad}
\cod{3/2*\aaa+\bbb/2}{\begin{tabular}{c}$1+a$\\$\downarrow$\end{tabular}}
\cod{\mmm}{\qquad\qquad(n+1)m+a}
\cod{3*\aaa}{nm+4a\qquad}
\cod{5/2*\aaa+\bbb/2}{\qquad 1+2a}
\gbox{1}{0}{\aaa}{1}
\gbox{1}{\aaa}{2*\aaa}{2} 
\gbox{1}{2*\aaa}{3/2*\aaa+\bbb/2}{3} \rper{1}{3/2*\aaa+\bbb/2}
\gbox{1}{3/2*\aaa+\bbb/2}{5/2*\aaa+\bbb/2}{0} \lper{1}{3/2*\aaa+\bbb/2}
\gbox{2}{0}{\aaa}{0}
\gbox{2}{\aaa}{2*\aaa}{1}
\gbox{2}{2*\aaa}{3*\aaa}{2} 
\gbox{2}{3*\aaa}{5/2*\aaa+\bbb/2}{3} \rper{2}{5/2*\aaa+\bbb/2}
\gbox{3}{0}{\bbb-\aaa}{2}
\gbox{3}{\bbb-\aaa}{\bbb}{0}
\gbox{3}{\bbb}{\mmm}{1}
\gbox{3}{\mmm}{\mmm+\aaa}{2}
\gbox{4}{\aaa/2+\bbb/2}{3/2*\aaa+\bbb/2}{0} \lper{4}{\aaa/2+\bbb/2}
\gbox{4}{3/2*\aaa+\bbb/2}{5/2*\aaa+\bbb/2}{1}
\end{tikzpicture}\tag{$n$}
\end{align}
Due to $f(x-1)=0$ for $(n+1)m\leq x<1+a$, the box $f(x)=2$ is stretched to the length $a$ and the new box $f(x)=3$ occurs. They trim the next box $f(x)=0$, but the stretched box $f(x-a)=2$ and the newly occurred box $f(x-a)=3$ stretches it.
\begin{align}
&\begin{tikzpicture}
\node at (\bbb/2-3/2*\aaa-.2,0.22) [left] {$f(x-1)$};
\node at (\bbb/2-3/2*\aaa-.2,0.72) [left] {$f(x-b)$};
\node at (\bbb/2-3/2*\aaa-.2,1.22) [left] {$f(x-a)$};
\node at (\bbb/2-3/2*\aaa-.2,1.72) [left] {$f(x)$};
\node at (\bbb/2-3/2*\aaa-.2,2.22) [left] {$x\ $};
\cod{0}{1+2a\quad}
\cod{\bbb/2-\aaa/2}{\begin{tabular}{c}$(n+1)m+2a$\\$\downarrow$\end{tabular}}
\cod{\bbb-\aaa}{\qquad 1+a+b}
\cod{\aaa}{1+3a}
\cod{\bbb}{\quad 1+2a+b}
\cod{2*\aaa}{1+4a}
\cod{\mmm}{\quad 1+3a+b}
\gbox{1}{0}{\aaa}{1}
\gbox{1}{\aaa}{\bbb}{2}
\gbox{1}{\bbb}{\mmm}{0}
\gbox{2}{0}{\aaa}{0} \lper{2}{0}
\gbox{2}{\aaa}{2*\aaa}{1}
\gbox{2}{2*\aaa}{\mmm}{2}
\gbox{3}{\bbb/2-3/2*\aaa}{\bbb/2-\aaa/2}{2}
\gbox{3}{\bbb/2-\aaa/2}{\bbb-\aaa}{3} \rper{3}{\bbb-\aaa}
\gbox{3}{\bbb-\aaa}{\bbb}{0} \lper{3}{\bbb-\aaa}
\gbox{3}{\bbb}{\mmm}{1}
\gbox{4}{0}{\bbb-\aaa}{2}
\gbox{4}{\bbb-\aaa}{\bbb}{0}
\gbox{4}{\bbb}{\mmm}{1}
\end{tikzpicture}\tag{$n+1$}
\end{align}
Again we have $f(x-1)=f(x-b)$ for $1+a+b\leq x<1+mn+2a$.

\begin{align}
&\begin{tikzpicture}
\node at (-.2,0.22) [left] {$f(x-1)$};
\node at (-.2,0.72) [left] {$f(x-b)$};
\node at (-.2,1.22) [left] {$f(x-a)$};
\node at (-.2,1.72) [left] {$f(x)$};
\node at (-.2,2.22) [left] {$x\ $};
\cod{0}{\begin{tabular}{c}$1+nm+2a$\\$\downarrow$\end{tabular}}
\cod{\bbb-\aaa}{\quad 1+(n+1)m}
\cod{\aaa}{1+nm+3a\qquad\qquad}
\cod{\aaa/2+\bbb/2}{\begin{tabular}{c}$2+a$\\$\downarrow$\end{tabular}}
\cod{\bbb}{\qquad\qquad\qquad 1+(n+1)m+a}
\gbox{1}{0}{\aaa}{1}
\gbox{1}{\aaa}{2*\aaa}{2} 
\gbox{2}{0}{\aaa}{0}
\gbox{2}{\aaa}{2*\aaa}{1}
\gbox{3}{0}{\bbb-\aaa}{2}
\gbox{3}{\bbb-\aaa}{\bbb}{0}
\gbox{3}{\bbb}{\mmm}{1}
\gbox{4}{0}{\aaa}{2}
\gbox{4}{\aaa}{\aaa/2+\bbb/2}{3} \rper{4}{\aaa/2+\bbb/2}
\gbox{4}{\aaa/2+\bbb/2}{3/2*\aaa+\bbb/2}{0} \lper{4}{\aaa/2+\bbb/2}
\end{tikzpicture}\tag{$2n+1$}
\end{align}
Thus $ f(x) =\mex\{f( x-s) \mid s\in S\}$ for $ 1\leq x< 2+a$.\\

\item For $W_n^2$ where $3a-b\leq\lambda=1-a-(n-1)m<2b-a$,
\begin{align*}
&\begin{tikzpicture}
\node at (\aaa-\bbc-.2,0.22) [left] {$\varphi(x-1)$};
\node at (\aaa-\bbc-.2,0.72) [left] {$\varphi(x-b)$};
\node at (\aaa-\bbc-.2,1.22) [left] {$\varphi(x-a)$};
\node at (\aaa-\bbc-.2,1.72) [left] {$\varphi(x)$};
\node at (\aaa-\bbc-.2,2.22) [left] {$x\ $};
\cod{\aaa-\bbc/2}{1}
\cod{2*\aaa-\bbc}{\quad(n-1)m+3a}
\cod{\aaa}{\qquad nm+a}
\cod{2*\aaa-\bbc/2}{\begin{tabular}{c}$1+a\qquad$\\$\downarrow$\end{tabular}}
\cod{\bbc}{\quad\  nm+b}
\cod{3*\aaa-\bbc}{\begin{tabular}{c}$(n-1)m+4a$\quad\\$\downarrow$\end{tabular}}
\cod{\aaa+\bbc/2}{\quad 1+b}
\cod{2*\aaa}{\begin{tabular}{c}\quad$nm+2a$\\$\downarrow$\end{tabular}}
\cod{3*\aaa-\bbc/2}{1+2a\qquad\quad}
\cod{\mmc}{\begin{tabular}{c}$(n+1)m$\quad\ \\$\downarrow$\end{tabular}}
\cod{4*\aaa-\bbc}{\begin{tabular}{c}$\qquad\qquad\quad (n-1)m+5a$\\$\downarrow$\end{tabular}}
\cod{2*\aaa+\bbc/2}{\qquad\quad 1+a+b}
\gbox{1}{\aaa-\bbc/2}{2*\aaa-\bbc}{2}
\gbox{1}{2*\aaa-\bbc}{\aaa}{3}
\gbox{1}{\aaa}{2*\aaa-\bbc/2}{1}
\gbox{1}{2*\aaa-\bbc/2}{\bbc}{0}
\gbox{1}{\bbc}{3*\aaa-\bbc}{3}
\gbox{1}{3*\aaa-\bbc}{\aaa+\bbc/2}{2}
\gbox{1}{\aaa+\bbc/2}{3*\aaa-\bbc/2}{0}
\gbox{1}{3*\aaa-\bbc/2}{\mmc}{1}
\gbox{1}{\mmc}{2*\aaa+\bbc/2}{0}
\gbox{2}{\aaa-\bbc}{2*\aaa-\bbc}{1}
\gbox{2}{2*\aaa-\bbc}{\aaa}{2}
\gbox{2}{\aaa}{2*\aaa-\bbc/2}{0}
\gbox{2}{2*\aaa-\bbc/2}{3*\aaa-\bbc}{2}
\gbox{2}{3*\aaa-\bbc}{2*\aaa}{3}
\gbox{2}{2*\aaa}{3*\aaa-\bbc/2}{1}
\gbox{2}{3*\aaa-\bbc/2}{\mmc}{0}
\gbox{2}{\mmc}{4*\aaa-\bbc}{3}
\gbox{2}{4*\aaa-\bbc}{2*\aaa+\bbc/2}{2}
\gbox{3}{0}{\aaa}{1}
\gbox{3}{\aaa}{\bbc}{2}
\gbox{3}{\bbc}{\aaa+\bbc/2}{0}
\gbox{3}{\aaa+\bbc/2}{2*\aaa}{2}
\gbox{3}{2*\aaa}{\mmc}{3}
\gbox{3}{\mmc}{2*\aaa+\bbc/2}{1}
\gbox{4}{\aaa-\bbc/2}{2*\aaa-\bbc/2}{0}
\gbox{4}{2*\aaa-\bbc/2}{3*\aaa-\bbc/2}{1}
\gbox{4}{3*\aaa-\bbc/2}{2*\aaa+\bbc/2}{2}
\end{tikzpicture}
\end{align*}

\newpage
\item For $W_n^3$ where $b<\lambda=1-a-(n-1)m<2a<2b-a$,
\begin{align*}
&\begin{tikzpicture}
\node at (\aaa-\bbd-.2,0.22) [left] {$\varphi(x-1)$};
\node at (\aaa-\bbd-.2,0.72) [left] {$\varphi(x-b)$};
\node at (\aaa-\bbd-.2,1.22) [left] {$\varphi(x-a)$};
\node at (\aaa-\bbd-.2,1.72) [left] {$\varphi(x)$};
\node at (\aaa-\bbd-.2,2.22) [left] {$x\ $};
\cod{\aaa-\bbd/2}{1}
\cod{2*\aaa-\bbd}{\quad\qquad(n-1)m+3a}
\cod{\aaa}{\begin{tabular}{c}$nm+a$\\$\downarrow$\end{tabular}}
\cod{2*\aaa-\bbd/2}{1+a\quad}
\cod{3*\aaa-\bbd}{\begin{tabular}{c}\quad$(n-1)m+4a$\\$\downarrow$\end{tabular}}
\cod{\bbd}{nm+b\ }
\cod{\aaa+\bbd/2}{\quad 1+b}
\cod{2*\aaa}{\begin{tabular}{c}$nm+2a$\\$\downarrow$\end{tabular}}
\cod{3*\aaa-\bbd/2}{1+2a\ }
\cod{\mmd}{(n+1)m\qquad\quad\ \ }
\cod{2*\aaa+\bbd/2}{1+a+b}
\gbox{1}{\aaa-\bbd/2}{2*\aaa-\bbd}{2}
\gbox{1}{2*\aaa-\bbd}{\aaa}{3}
\gbox{1}{\aaa}{2*\aaa-\bbd/2}{1}
\gbox{1}{2*\aaa-\bbd/2}{\bbd}{0}
\gbox{1}{\bbd}{\aaa+\bbd/2}{2}
\gbox{1}{\aaa+\bbd/2}{3*\aaa-\bbd/2}{0}
\gbox{1}{3*\aaa-\bbd/2}{\mmd}{1}
\gbox{1}{\mmd}{2*\aaa+\bbd/2}{0}
\gbox{2}{\aaa-\bbd}{2*\aaa-\bbd}{1}
\gbox{2}{2*\aaa-\bbd}{\aaa}{2}
\gbox{2}{\aaa}{2*\aaa-\bbd/2}{0}
\gbox{2}{2*\aaa-\bbd/2}{3*\aaa-\bbd}{2}
\gbox{2}{3*\aaa-\bbd}{2*\aaa}{3}
\gbox{2}{2*\aaa}{3*\aaa-\bbd/2}{1}
\gbox{2}{3*\aaa-\bbd/2}{\mmd}{0}
\gbox{2}{\mmd}{2*\aaa+\bbd/2}{2}
\gbox{3}{0}{\aaa}{1}
\gbox{3}{\aaa}{\bbd}{2}
\gbox{3}{\bbd}{\aaa+\bbd/2}{0}
\gbox{3}{\aaa+\bbd/2}{2*\aaa}{2}
\gbox{3}{2*\aaa}{\mmd}{3}
\gbox{3}{\mmd}{2*\aaa+\bbd/2}{1}
\gbox{4}{\aaa-\bbd/2}{2*\aaa-\bbd/2}{0}
\gbox{4}{2*\aaa-\bbd/2}{3*\aaa-\bbd/2}{1}
\gbox{4}{3*\aaa-\bbd/2}{2*\aaa+\bbd/2}{2}
\end{tikzpicture}
\end{align*}

\item For $W_n^4$ where $2a\leq\lambda=1-a-(n-1)m<2b-a$,
\begin{align*}
&\begin{tikzpicture}
\node at (-.2,0.22) [left] {$\varphi(x-1)$};
\node at (-.2,0.72) [left] {$\varphi(x-b)$};
\node at (-.2,1.22) [left] {$\varphi(x-a)$};
\node at (-.2,1.72) [left] {$\varphi(x)$};
\node at (-.2,2.22) [left] {$x\ $};
\cod{\aaa/2}{1}
\cod{\aaa}{nm+a}
\cod{3*\aaa/2}{1+a\qquad\ }
\cod{\bbd}{\qquad\  nm+b}
\cod{\aaa/2+\bbd}{\quad 1+b}
\cod{2*\aaa}{\begin{tabular}{c}$nm+2a$\\$\downarrow$\end{tabular}}
\cod{5*\aaa/2}{1+2a\qquad\ }
\cod{\mmd}{\qquad\quad (n+1)m}
\cod{3*\aaa/2+\bbd}{1+a+b}
\gbox{1}{\aaa/2}{\aaa}{3}
\gbox{1}{\aaa}{3*\aaa/2}{1}
\gbox{1}{3*\aaa/2}{\bbd}{0}
\gbox{1}{\bbd}{\aaa/2+\bbd}{2}
\gbox{1}{\aaa/2+\bbd}{5*\aaa/2}{0}
\gbox{1}{5*\aaa/2}{\mmd}{1}
\gbox{1}{\mmd}{3/2*\aaa+\bbd}{0}
\gbox{2}{2*\aaa-\bbd}{\aaa}{2}
\gbox{2}{\aaa}{3*\aaa/2}{0}
\gbox{2}{3*\aaa/2}{2*\aaa}{3}
\gbox{2}{2*\aaa}{5*\aaa/2}{1}
\gbox{2}{5*\aaa/2}{\mmd}{0}
\gbox{2}{\mmd}{3*\aaa/2+\bbd}{2}
\gbox{3}{0}{\aaa}{1}
\gbox{3}{\aaa}{\bbd}{2}
\gbox{3}{\bbd}{\aaa/2+\bbd}{0}
\gbox{3}{\aaa/2+\bbd}{\mmd}{3}
\gbox{3}{\mmd}{3*\aaa/2+\bbd}{1}
\gbox{4}{\aaa/2}{3*\aaa/2}{0}
\gbox{4}{3*\aaa/2}{5*\aaa/2}{1}
\gbox{4}{5*\aaa/2}{3*\aaa/2+\bbd}{2}
\end{tikzpicture}
\end{align*}

\newpage
\item For $B_n$.
\begin{enumerate}
\item \label{b1} For $B_n^1$ where $ 2b-a\leq \lambda=1-a-(n-1)m < 3a-b( \leq 2a\leq a+b)$,
\begin{align}
&\begin{tikzpicture}
\node at (-.2,0.22) [left] {$f(x-1)$};
\node at (-.2,0.72) [left] {$f(x-b)$};
\node at (-.2,1.22) [left] {$f(x-a)$};
\node at (-.2,1.72) [left] {$f(x)$};
\node at (-.2,2.22) [left] {$x\ $};
\cod{0}{\begin{tabular}{c}$(n-1)m+a$\\$\downarrow$\end{tabular}}
\cod{\bbb-\aaa}{\quad(n-1)m+b}
\cod{\aaa}{(n-1)m+2a\qquad}
\cod{\bbb}{nm}
\cod{\aaa+\bbb/2}{1}
\cod{2*\aaa}{\begin{tabular}{c}$(n-1)m+3a$\\$\downarrow$\end{tabular}}
\cod{\mmm}{nm+a}
\gbox{1}{0}{\aaa}{1}
\gbox{1}{\aaa}{\bbb}{2} 
\gbox{1}{\bbb}{\aaa+\bbb/2}{0}
\gbox{1}{\aaa+\bbb/2}{2*\aaa}{2}
\gbox{1}{2*\aaa}{\mmm}{3}
\gbox{2}{0}{\aaa}{0}
\gbox{2}{\aaa}{2*\aaa}{1}
\gbox{2}{2*\aaa}{\mmm}{2}
\gbox{3}{0}{\bbb-\aaa}{2}
\gbox{3}{\bbb-\aaa}{\bbb}{0}
\gbox{3}{\bbb}{\mmm}{1}
\gbox{4}{\aaa+\bbb/2}{2*\aaa+\bbb/2}{0} \lper{4}{\aaa+\bbb/2}
\end{tikzpicture}\tag{$n-1$}\\
&\begin{tikzpicture}
\node at (\aaa+\bbb/2-.2,0.22) [left] {$f(x-1)$};
\node at (\aaa+\bbb/2-.2,0.72) [left] {$f(x-b)$};
\node at (\aaa+\bbb/2-.2,1.22) [left] {$f(x-a)$};
\node at (\aaa+\bbb/2-.2,1.72) [left] {$f(x)$};
\node at (\aaa+\bbb/2-.2,2.22) [left] {$x\ $};
\cod{\mmm}{\begin{tabular}{c}$nm+a$\\$\downarrow$\end{tabular}}
\cod{2*\bbb}{nm+b\qquad\quad}
\cod{2*\aaa+\bbb/2}{\begin{tabular}{c}$1+a$\\$\downarrow$\end{tabular}}
\cod{\aaa+3/2*\bbb}{1+b\qquad}
\cod{3*\aaa}{\begin{tabular}{c}$\quad(n-1)m+4a$\\$\downarrow$\end{tabular}}
\cod{2*\aaa+\bbb}{nm+2a}
\cod{\aaa+2*\bbb}{\begin{tabular}{c}$(n+1)m$\\$\downarrow$\end{tabular}}
\cod{3*\aaa+\bbb/2}{\qquad\quad 1+2a}
\cod{2*\aaa+3/2*\bbb}{\begin{tabular}{c}$1+a+b$\\$\downarrow$\end{tabular}}
\gbox{1}{\mmm}{2*\aaa+\bbb/2}{1}
\gbox{1}{2*\aaa+\bbb/2}{\aaa+3/2*\bbb}{3} \rper{1}{\aaa+3/2*\bbb}
\gbox{1}{\aaa+3/2*\bbb}{2*\aaa+3/2*\bbb}{0} \lper{1}{\aaa+3/2*\bbb}
\gbox{2}{\mmm}{2*\aaa+\bbb/2}{0}
\gbox{2}{2*\aaa+\bbb/2}{3*\aaa}{2}
\gbox{2}{3*\aaa}{2*\aaa+\bbb}{3}
\gbox{2}{2*\aaa+\bbb}{3*\aaa+\bbb/2}{1}
\gbox{2}{3*\aaa+\bbb/2}{2*\aaa+3/2*\bbb}{3} \rper{2}{2*\aaa+3/2*\bbb}
\gbox{3}{\mmm}{2*\bbb}{2}
\gbox{3}{2*\bbb}{\aaa+3/2*\bbb}{0}
\gbox{3}{\aaa+3/2*\bbb}{2*\aaa+\bbb}{2}
\gbox{3}{2*\aaa+\bbb}{\aaa+2*\bbb}{3}
\gbox{3}{\aaa+2*\bbb}{2*\aaa+3/2*\bbb}{1}
\gbox{4}{\aaa+\bbb/2}{2*\aaa+\bbb/2}{0} \lper{4}{\aaa+\bbb/2}
\gbox{4}{2*\aaa+\bbb/2}{3*\aaa+\bbb/2}{1}
\gbox{4}{3*\aaa+\bbb/2}{2*\aaa+3/2*\bbb}{2}
\end{tikzpicture}\tag{$n$}
\end{align}
In the diagram $\langle n \rangle$, the boxes $f(x-a)=0$ and $f(x-1)=0$ end at the same time $x=1+a$, and the box $f(x-b)=0$ begins at $x=nm+b$. $2b-a\leq \lambda$ implies $nm+b\leq 1+a$, so the box $f(x-b)=0$ begins before the boxes $f(x-a)=0$ and $f(x-1)=0$ end.
After the box $f(x-b)=0$ ends $ 1+b$, all of $ f( x-s)$ happen to be positive and then the box $f(x)=0$ occurs and lasts until the end of the boxes $f(x-a)>0$. The timing is $x=1+a+b$, the same as when the box $f(x-1)=2$ ends.

\begin{align}
&\begin{tikzpicture}
\node at (-.7,0.22) [left] {$f(x-1)$};
\node at (-.7,0.72) [left] {$f(x-b)$};
\node at (-.7,1.22) [left] {$f(x-a)$};
\node at (-.7,1.72) [left] {$f(x)$};
\node at (-.7,2.22) [left] {$x\ $};
\cod{0}{1+a+b}
\cod{\bbb-\aaa}{1+2b}
\cod{\aaa}{1+2a+b\qquad}
\cod{\bbb}{\qquad 1+a+2b}
\cod{2*\aaa}{1+3a+b\qquad}
\cod{\mmm}{\qquad 1+2a+2b}
\gbox{1}{0}{\aaa}{1}
\gbox{1}{\aaa}{\bbb}{2}
\gbox{1}{\bbb}{\aaa+\bbb}{0}
\gbox{2}{0}{\aaa}{0}
\gbox{2}{\aaa}{2*\aaa}{1}
\gbox{2}{2*\aaa}{\aaa+\bbb}{2}
\gbox{3}{0}{\bbb-\aaa}{3} \rper{3}{\bbb-\aaa}
\gbox{3}{\bbb-\aaa}{\bbb}{0} \lper{3}{\bbb-\aaa}
\gbox{3}{\bbb}{\aaa+\bbb}{1}
\gbox{4}{0}{\aaa}{0}
\gbox{4}{\aaa}{2*\aaa}{1}
\gbox{4}{2*\aaa}{\aaa+\bbb}{2}
\end{tikzpicture}\tag{$n+1$}
\end{align}
Similarly we have
\begin{align*}
f(x-1)&=f(x-(a+b)-1) &&(\text{the local period }a+b)\\
      &=f(x-(1+b)-a)\\
      &=f(x-a).      &&(\text{the period }1+b)
\end{align*}
for $1+a+b\leq x<1+nm+2a$.

\begin{align}
&\begin{tikzpicture}
\node at (-.6,0.22) [left] {$f(x-1)$};
\node at (-.6,0.72) [left] {$f(x-b)$};
\node at (-.6,1.22) [left] {$f(x-a)$};
\node at (-.6,1.72) [left] {$f(x)$};
\node at (-.6,2.22) [left] {$x\ $};
\cod{0}{1+nm}
\cod{\bbb-\aaa}{\begin{tabular}{c}$1+(n-1)m+2b$\\$\downarrow$\end{tabular}}
\cod{\aaa-\bbb/2}{2}
\cod{2*\aaa-\bbb}{1\!+\!(n\!-\!1)m\!+\!3a\ }
\cod{\aaa}{\begin{tabular}{c}$1+nm+a$\qquad\ \\$\downarrow$\end{tabular}}
\cod{\bbb}{\begin{tabular}{c}\qquad\ $1+nm+b$\\$\downarrow$\end{tabular}}
\cod{2*\aaa-\bbb/2}{\qquad 2+a}
\cod{\aaa+\bbb/2}{2+b}
\cod{2*\aaa}{1+nm+2a}
\gbox{1}{0}{\aaa}{1}
\gbox{1}{\aaa}{\bbb}{2} 
\gbox{1}{\bbb}{\aaa+\bbb/2}{0}
\gbox{1}{\aaa+\bbb/2}{2*\aaa}{2}
\gbox{2}{0}{\aaa}{0}
\gbox{2}{\aaa}{2*\aaa}{1}
\gbox{3}{0}{\bbb-\aaa}{2}
\gbox{3}{\bbb-\aaa}{\bbb}{0}
\gbox{3}{\bbb}{\mmm}{1}
\gbox{4}{0}{\aaa-\bbb/2}{0}
\gbox{4}{\aaa-\bbb/2}{2*\aaa-\bbb}{2}
\gbox{4}{2*\aaa-\bbb}{\aaa}{3}
\gbox{4}{\aaa}{2*\aaa-\bbb/2}{1}
\gbox{4}{2*\aaa-\bbb/2}{\aaa+\bbb/2}{3} \rper{4}{\aaa+\bbb/2}
\gbox{4}{\aaa+\bbb/2}{2*\aaa+\bbb/2}{0} \lper{4}{\aaa+\bbb/2}
\end{tikzpicture}\tag{$2n$}
\end{align}

Thus $ f(x) =\mex\{f( x-s) \mid s\in S\}$ for $ 1\leq x< 2+b$.

\newpage
\item For $B_n^2$ where $ \max( 2b-a,3a-b) \leq \lambda=1-a-(n-1)m < 2a( \leq a+b)$,
\begin{align}
&\begin{tikzpicture}
\node at (-.2,0.22) [left] {$f(x-1)$};
\node at (-.2,0.72) [left] {$f(x-b)$};
\node at (-.2,1.22) [left] {$f(x-a)$};
\node at (-.2,1.72) [left] {$f(x)$};
\node at (-.2,2.22) [left] {$x\ $};
\cod{0}{\begin{tabular}{c}$(n-1)m+a$\\$\downarrow$\end{tabular}}
\cod{\bbb-\aaa}{\quad(n-1)m+b}
\cod{\aaa}{(n-1)m+2a\qquad}
\cod{\bbb}{nm}
\cod{5/2*\aaa-\bbb/2}{1}
\cod{2*\aaa}{\begin{tabular}{c}$(n-1)m+3a$\\$\downarrow$\end{tabular}}
\cod{\mmm}{nm+a}
\gbox{1}{0}{\aaa}{1}
\gbox{1}{\aaa}{\bbb}{2} 
\gbox{1}{\bbb}{5/2*\aaa-\bbb/2}{0}
\gbox{1}{5/2*\aaa-\bbb/2}{2*\aaa}{2}
\gbox{1}{2*\aaa}{\mmm}{3}
\gbox{2}{0}{\aaa}{0}
\gbox{2}{\aaa}{2*\aaa}{1}
\gbox{2}{2*\aaa}{\mmm}{2}
\gbox{3}{0}{\bbb-\aaa}{2}
\gbox{3}{\bbb-\aaa}{\bbb}{0}
\gbox{3}{\bbb}{\mmm}{1}
\gbox{4}{5/2*\aaa-\bbb/2}{2*\aaa+\bbb/2}{0} \lper{4}{5/2*\aaa-\bbb/2}
\end{tikzpicture}\tag{$n-1$}\\
&\begin{tikzpicture}
\node at (5/2*\aaa-\bbb/2-.2,0.22) [left] {$f(x-1)$};
\node at (5/2*\aaa-\bbb/2-.2,0.72) [left] {$f(x-b)$};
\node at (5/2*\aaa-\bbb/2-.2,1.22) [left] {$f(x-a)$};
\node at (5/2*\aaa-\bbb/2-.2,1.72) [left] {$f(x)$};
\node at (5/2*\aaa-\bbb/2-.2,2.22) [left] {$x\ $};
\cod{\mmm}{\begin{tabular}{c}$nm+a$\\$\downarrow$\end{tabular}}
\cod{2*\bbb}{nm+b\quad}
\cod{7/2*\aaa-\bbb/2}{1+a}
\cod{3*\aaa}{\begin{tabular}{c}$(n-1)m+4a$\\$\downarrow$\end{tabular}}
\cod{5/2*\aaa+\bbb/2}{1+b\ }
\cod{2*\aaa+\bbb}{\qquad\quad nm+2a}
\cod{\aaa+2*\bbb}{\begin{tabular}{c}$(n+1)m$\\$\downarrow$\end{tabular}}
\cod{9/2*\aaa-\bbb/2}{1+2a\quad}
\cod{4*\aaa}{\begin{tabular}{c}$\quad(n-1)m+5a$\\$\downarrow$\end{tabular}}
\cod{7/2*\aaa+\bbb/2}{\qquad 1+a+b}
\gbox{1}{\mmm}{7/2*\aaa-\bbb/2}{1}
\gbox{1}{7/2*\aaa-\bbb/2}{3*\aaa}{3}
\gbox{1}{3*\aaa}{5/2*\aaa+\bbb/2}{2} \rper{1}{5/2*\aaa+\bbb/2}
\gbox{1}{5/2*\aaa+\bbb/2}{7/2*\aaa+\bbb/2}{0} \lper{1}{5/2*\aaa+\bbb/2}
\gbox{2}{\mmm}{7/2*\aaa-\bbb/2}{0}
\gbox{2}{7/2*\aaa-\bbb/2}{3*\aaa}{2}
\gbox{2}{3*\aaa}{2*\aaa+\bbb}{3}
\gbox{2}{2*\aaa+\bbb}{9/2*\aaa-\bbb/2}{1}
\gbox{2}{9/2*\aaa-\bbb/2}{4*\aaa}{3}
\gbox{2}{4*\aaa}{7/2*\aaa+\bbb/2}{2} \rper{2}{7/2*\aaa+\bbb/2}
\gbox{3}{\mmm}{2*\bbb}{2}
\gbox{3}{2*\bbb}{5/2*\aaa+\bbb/2}{0}
\gbox{3}{5/2*\aaa+\bbb/2}{2*\aaa+\bbb}{2}
\gbox{3}{2*\aaa+\bbb}{\aaa+2*\bbb}{3}
\gbox{3}{\aaa+2*\bbb}{7/2*\aaa+\bbb/2}{1}
\gbox{4}{5/2*\aaa-\bbb/2}{7/2*\aaa-\bbb/2}{0} \lper{4}{5/2*\aaa-\bbb/2}
\gbox{4}{7/2*\aaa-\bbb/2}{9/2*\aaa-\bbb/2}{1}
\gbox{4}{9/2*\aaa-\bbb/2}{7/2*\aaa+\bbb/2}{2}
\end{tikzpicture}\tag{$n$}
\end{align}
$ x=1$ is a little backward compared to (\ref{b1}), so that the box $f(x)=0$ in the diagram $\langle n-1\rangle$ is a little longer and the box $f(x)=2$ newly occurs at the end of the loop due to the box $f(x-b)=0$. However it makes no change to other values afterward in the rows $f(x-a)$ and $f(x-b)$.

\begin{align}
&\begin{tikzpicture}
\node at (-1,0.22) [left] {$f(x-1)$};
\node at (-1,0.72) [left] {$f(x-b)$};
\node at (-1,1.22) [left] {$f(x-a)$};
\node at (-1,1.72) [left] {$f(x)$};
\node at (-1,2.22) [left] {$x\ $};
\cod{0}{1+a+b\qquad}
\cod{\bbb/2-\aaa/2}{\begin{tabular}{c}$nm+3a$\\$\downarrow$\end{tabular}}
\cod{\bbb-\aaa}{\quad 1+2b}
\cod{\aaa}{1+2a+b\qquad}
\cod{\bbb}{\qquad 1+a+2b}
\cod{2*\aaa}{1+3a+b\qquad}
\cod{\mmm}{\qquad 1+2a+2b}
\gbox{1}{0}{\aaa}{1}
\gbox{1}{\aaa}{\bbb}{2}
\gbox{1}{\bbb}{\aaa+\bbb}{0}
\gbox{2}{0}{\aaa}{0}
\gbox{2}{\aaa}{2*\aaa}{1}
\gbox{2}{2*\aaa}{\aaa+\bbb}{2}
\gbox{3}{0}{\bbb/2-\aaa/2}{3}
\gbox{3}{\bbb/2-\aaa/2}{\bbb-\aaa}{2} \rper{3}{\bbb-\aaa}
\gbox{3}{\bbb-\aaa}{\bbb}{0} \lper{3}{\bbb-\aaa}
\gbox{3}{\bbb}{\aaa+\bbb}{1}
\gbox{4}{0}{\aaa}{0}
\gbox{4}{\aaa}{2*\aaa}{1}
\gbox{4}{2*\aaa}{\aaa+\bbb}{2}
\end{tikzpicture}\tag{$n+1$}
\end{align}
Again we have $ f( x-1) =f( x-a)$ for $1+a+b\leq x< 1+nm +2a$.

\begin{align}
&\begin{tikzpicture}
\node at (-.6,0.22) [left] {$f(x-1)$};
\node at (-.6,0.72) [left] {$f(x-b)$};
\node at (-.6,1.22) [left] {$f(x-a)$};
\node at (-.6,1.72) [left] {$f(x)$};
\node at (-.6,2.22) [left] {$x\ $};
\cod{0}{1+nm}
\cod{\bbb-\aaa}{\begin{tabular}{c}$1+(n-1)m+2b$\qquad\qquad\qquad\qquad\\$\downarrow$\end{tabular}}
\cod{5/2*\aaa-3/2*\bbb}{2}
\cod{2*\aaa-\bbb}{\begin{tabular}{c}\qquad $1+(n-1)m+3a$\\$\downarrow$\end{tabular}}
\cod{\aaa}{1+nm+a}
\cod{\bbb}{\begin{tabular}{c}\quad$1+nm+b$\\$\downarrow$\end{tabular}}
\cod{7/2*\aaa-3/2*\bbb}{2+a\quad}
\cod{3*\aaa-\bbb}{\begin{tabular}{c}\qquad\qquad\qquad $1+(n-1)m+4a$\\$\downarrow$\end{tabular}}
\cod{5/2*\aaa-\bbb/2}{2+b\ }
\cod{2*\aaa}{\qquad\qquad\quad 1+nm+2a}
\gbox{1}{0}{\aaa}{1}
\gbox{1}{\aaa}{\bbb}{2}
\gbox{1}{\bbb}{5/2*\aaa-\bbb/2}{0}
\gbox{1}{5/2*\aaa-\bbb/2}{2*\aaa}{2}
\gbox{2}{0}{\aaa}{0}
\gbox{2}{\aaa}{2*\aaa}{1}
\gbox{3}{0}{\bbb-\aaa}{2}
\gbox{3}{\bbb-\aaa}{\bbb}{0}
\gbox{3}{\bbb}{\mmm}{1}
\gbox{4}{0}{5/2*\aaa-3/2*\bbb}{0}
\gbox{4}{5/2*\aaa-3/2*\bbb}{2*\aaa-\bbb}{2}
\gbox{4}{2*\aaa-\bbb}{\aaa}{3}
\gbox{4}{\aaa}{7/2*\aaa-3/2*\bbb}{1}
\gbox{4}{7/2*\aaa-3/2*\bbb}{3*\aaa-\bbb}{3}
\gbox{4}{3*\aaa-\bbb}{5/2*\aaa-\bbb/2}{2} \rper{4}{5/2*\aaa-\bbb/2}
\gbox{4}{5/2*\aaa-\bbb/2}{7/2*\aaa-\bbb/2}{0} \lper{4}{5/2*\aaa-\bbb/2}
\end{tikzpicture}\tag{$2n$}
\end{align}

Thus $ f(x) =\mex\{f( x-s) \mid s\in S\}$ for $ 1\leq x< 2+b$.

\newpage
\item For $B_n^3$ where $ \max( 2b-a,2a) \leq \lambda=1-a-(n-1)m < a+b$,
\begin{align}
&\begin{tikzpicture}
\node at (-.2,0.22) [left] {$f(x-1)$};
\node at (-.2,0.72) [left] {$f(x-b)$};
\node at (-.2,1.22) [left] {$f(x-a)$};
\node at (-.2,1.72) [left] {$f(x)$};
\node at (-.2,2.22) [left] {$x\ $};
\cod{0}{\begin{tabular}{c}$(n-1)m+a$\\$\downarrow$\end{tabular}}
\cod{\bbb-\aaa}{\quad(n-1)m+b}
\cod{\aaa}{(n-1)m+2a\qquad}
\cod{\bbb}{nm}
\cod{2*\aaa}{\begin{tabular}{c}$(n-1)m+3a$\\$\downarrow$\end{tabular}}
\cod{3/2*\aaa+\bbb/2}{1}
\cod{\mmm}{\qquad nm+a}
\gbox{1}{0}{\aaa}{1}
\gbox{1}{\aaa}{\bbb}{2} 
\gbox{1}{\bbb}{3/2*\aaa+\bbb/2}{0}
\gbox{1}{3/2*\aaa+\bbb/2}{\mmm}{3}
\gbox{2}{0}{\aaa}{0}
\gbox{2}{\aaa}{2*\aaa}{1}
\gbox{2}{2*\aaa}{\mmm}{2}
\gbox{3}{0}{\bbb-\aaa}{2}
\gbox{3}{\bbb-\aaa}{\bbb}{0}
\gbox{3}{\bbb}{\mmm}{1}
\gbox{4}{3/2*\aaa+\bbb/2}{5/2*\aaa+\bbb/2}{0} \lper{4}{3/2*\aaa+\bbb/2}
\end{tikzpicture}\tag{$n-1$}\\
&\begin{tikzpicture}
\node at (3/2*\aaa+\bbb/2-.3,0.22) [left] {$f(x-1)$};
\node at (3/2*\aaa+\bbb/2-.3,0.72) [left] {$f(x-b)$};
\node at (3/2*\aaa+\bbb/2-.3,1.22) [left] {$f(x-a)$};
\node at (3/2*\aaa+\bbb/2-.3,1.72) [left] {$f(x)$};
\node at (3/2*\aaa+\bbb/2-.3,2.22) [left] {$x\ $};
\cod{\mmm}{nm+a\quad}
\cod{2*\bbb}{\quad nm+b}
\cod{5/2*\aaa+\bbb/2}{1+a}
\cod{2*\aaa+\bbb}{\begin{tabular}{c}$nm+2a$\\$\downarrow$\end{tabular}}
\cod{3/2*\aaa+3/2*\bbb}{1+b\ }
\cod{\aaa+2*\bbb}{\qquad\quad (n+1)m}
\cod{7/2*\aaa+\bbb/2}{1+2a}
\cod{5/2*\aaa+3/2*\bbb}{\quad 1+a+b}
\gbox{1}{\mmm}{5/2*\aaa+\bbb/2}{1}
\gbox{1}{5/2*\aaa+\bbb/2}{3/2*\aaa+3/2*\bbb}{2} \rper{1}{3/2*\aaa+3/2*\bbb}
\gbox{1}{3/2*\aaa+3/2*\bbb}{5/2*\aaa+3/2*\bbb}{0} \lper{1}{3/2*\aaa+3/2*\bbb}
\gbox{2}{\mmm}{5/2*\aaa+\bbb/2}{0}
\gbox{2}{5/2*\aaa+\bbb/2}{2*\aaa+\bbb}{3}
\gbox{2}{2*\aaa+\bbb}{7/2*\aaa+\bbb/2}{1}
\gbox{2}{7/2*\aaa+\bbb/2}{5/2*\aaa+3/2*\bbb}{2} \rper{2}{5/2*\aaa+3/2*\bbb}
\gbox{3}{\mmm}{2*\bbb}{2}
\gbox{3}{2*\bbb}{3/2*\aaa+3/2*\bbb}{0}
\gbox{3}{3/2*\aaa+3/2*\bbb}{\aaa+2*\bbb}{3}
\gbox{3}{\aaa+2*\bbb}{5/2*\aaa+3/2*\bbb}{1}
\gbox{4}{3/2*\aaa+\bbb/2}{5/2*\aaa+\bbb/2}{0} \lper{4}{3/2*\aaa+\bbb/2}
\gbox{4}{5/2*\aaa+\bbb/2}{7/2*\aaa+\bbb/2}{1}
\gbox{4}{7/2*\aaa+\bbb/2}{5/2*\aaa+3/2*\bbb}{2}
\end{tikzpicture}\tag{$n$}
\end{align}
$x=1$ is more backward so that the change in the row $f(x)$ begins with the value $3$ (no $2$) and ends with $2$ (no $3$).

\begin{align}
&\begin{tikzpicture}
\node at (-1,0.22) [left] {$f(x-1)$};
\node at (-1,0.72) [left] {$f(x-b)$};
\node at (-1,1.22) [left] {$f(x-a)$};
\node at (-1,1.72) [left] {$f(x)$};
\node at (-1,2.22) [left] {$x\ $};
\cod{0}{1+a+b\quad}
\cod{\bbb-\aaa}{\quad 1+2b}
\cod{\aaa}{1+2a+b\qquad}
\cod{\bbb}{\qquad 1+a+2b}
\cod{2*\aaa}{1+3a+b\qquad}
\cod{\mmm}{\qquad 1+2a+2b}
\gbox{1}{0}{\aaa}{1}
\gbox{1}{\aaa}{\bbb}{2}
\gbox{1}{\bbb}{\aaa+\bbb}{0}
\gbox{2}{0}{\aaa}{0}
\gbox{2}{\aaa}{2*\aaa}{1}
\gbox{2}{2*\aaa}{\aaa+\bbb}{2}
\gbox{3}{0}{\bbb-\aaa}{2} \rper{3}{\bbb-\aaa}
\gbox{3}{\bbb-\aaa}{\bbb}{0} \lper{3}{\bbb-\aaa}
\gbox{3}{\bbb}{\aaa+\bbb}{1}
\gbox{4}{0}{\aaa}{0}
\gbox{4}{\aaa}{2*\aaa}{1}
\gbox{4}{2*\aaa}{\aaa+\bbb}{2}
\end{tikzpicture}\tag{$n+1$}
\end{align}
Once again $f( x-1) =f( x-a)$ holds for $1+a+b\leq x< 1+nm +2a$.

\begin{align}
&\begin{tikzpicture}
\node at (-.6,0.22) [left] {$f(x-1)$};
\node at (-.6,0.72) [left] {$f(x-b)$};
\node at (-.6,1.22) [left] {$f(x-a)$};
\node at (-.6,1.72) [left] {$f(x)$};
\node at (-.6,2.22) [left] {$x\ $};
\cod{0}{1+nm}
\cod{\bbb-\aaa}{\begin{tabular}{c}$1+(n-1)m+2b$\\$\downarrow$\end{tabular}}
\cod{3/2*\aaa-\bbb/2}{2}
\cod{\aaa}{\begin{tabular}{c}$1+nm+a$\\$\downarrow$\end{tabular}}
\cod{\bbb}{1+nm+b}
\cod{2*\aaa}{1+nm+2a\qquad\qquad}
\cod{3/2*\aaa+\bbb/2}{\begin{tabular}{c}$2+b$\\$\downarrow$\end{tabular}}
\cod{\mmm}{\qquad\qquad 1+(n+1)m}
\gbox{1}{0}{\aaa}{1}
\gbox{1}{\aaa}{\bbb}{2}
\gbox{1}{\bbb}{3/2*\aaa+\bbb/2}{0}
\gbox{1}{3/2*\aaa+\bbb/2}{\mmm}{2}
\gbox{2}{0}{\aaa}{0}
\gbox{2}{\aaa}{2*\aaa}{1}
\gbox{2}{\aaa}{\mmm}{1}
\gbox{3}{0}{\bbb-\aaa}{2}
\gbox{3}{\bbb-\aaa}{\bbb}{0}
\gbox{3}{\bbb}{\mmm}{1}
\gbox{4}{0}{3/2*\aaa-\bbb/2}{0}
\gbox{4}{3/2*\aaa-\bbb/2}{\aaa}{3}
\gbox{4}{\aaa}{2*\aaa}{1}
\gbox{4}{2*\aaa}{3/2*\aaa+\bbb/2}{2} \rper{4}{3/2*\aaa+\bbb/2}
\gbox{4}{3/2*\aaa+\bbb/2}{5/2*\aaa+\bbb/2}{0} \lper{4}{3/2*\aaa+\bbb/2}
\end{tikzpicture}\tag{$2n$}
\end{align}

Thus $ f(x) =\mex\{f( x-s) \mid s\in S\}$ for $ 1\leq x< 2+b$.

\end{enumerate}

\end{itemize}

\end{landscape}

\begin{rem} \label{rem:border}

Letting $m=a+b$, for $(a,b)$ in $B_n^3$ where $ \max( 2b-a,2a) \leq \lambda=1-a-(n-1)m<m$, fix $n$ and $\frac{b}{a}$ and make $\lambda\uparrow m$, then $mn+a\to1$, so
\begin{align*}
&\begin{tikzpicture}[baseline=1.65cm]
\cod{0}{0} \cod{1.5}{a} \cod{3}{2a} \cod{19/8*1.5}{m}
\cod{19/8*1.5+\repw}{nm}
\cod{19/8*1.5+\repw+4/13*\haa}{1}
\cod{19/8*1.5+\repw+8/13*\haa}{nm+a}
\cod{19/8*1.5+\repw+12/13*\haa}{1+a}
\cod{19/8*1.5+\repw+15/13*\haa}{1+b}
\gbox{1}{0}{1.5}{0} \lper{1}{0}
\gbox{1}{1.5}{3}{1} \gbox{1}{3}{19/8*1.5}{2}
\node at (19/8*1.5-0.1,1.6) [right] {$\times n$};
\gbox{1}{19/8*1.5+\repw}{19/8*1.5+\repw+4/13*\haa}{0}
\gbox{1}{19/8*1.5+\repw+4/13*\haa}{19/8*1.5+\repw+8/13*\haa}{3}
\gbox{1}{19/8*1.5+\repw+8/13*\haa}{19/8*1.5+\repw+12/13*\haa}{1}
\gbox{1}{19/8*1.5+\repw+12/13*\haa}{19/8*1.5+\repw+15/13*\haa}{2}
\rper{1}{19/8*1.5+\repw+15/13*\haa}
\end{tikzpicture}\\
\to&\begin{tikzpicture}[baseline=1.65cm]
\cod{0}{0} \cod{1.5}{a} \cod{3}{2a} \cod{19/8*1.5}{m}
\cod{19/8*1.5+\repw}{nm}
\cod{19/8*1.5+\repw+1.5}{nm+a\qquad}
\cod{19/8*1.5+\repw+3}{nm+2a\qquad}
\cod{19/8*1.5+\repw+19/8*1.5}{\qquad\qquad (n+1)m}
\gbox{1}{0}{1.5}{0} \lper{1}{0}
\gbox{1}{1.5}{3}{1}
\gbox{1}{3}{19/8*1.5}{2}
\node at (19/8*1.5-.1,1.6) [right] {$\times n$};
\gbox{1}{19/8*1.5+\repw}{19/8*1.5+\repw+1.5}{0}
\gbox{1}{19/8*1.5+\repw+1.5}{19/8*1.5+\repw+3}{1}
\gbox{1}{19/8*1.5+\repw+3}{19/8*1.5+\repw+19/8*1.5}{2}
\rper{1}{19/8*1.5+\repw+19/8*1.5}
\end{tikzpicture}\\
=&\begin{tikzpicture}[baseline=1.65cm]
\cod{0}{0} \cod{1.5}{a} \cod{3}{2a} \cod{19/8*1.5}{m}
\gbox{1}{0}{1.5}{0} \lper{1}{0}
\gbox{1}{1.5}{3}{1}
\gbox{1}{3}{19/8*1.5}{2}
\rper{1}{19/8*1.5}
\end{tikzpicture}.
\end{align*}
For $(a,b)$ in $R_n^1$ ($ n\geq 1$) where $ b-a< \lambda=1-a-nm \leq a$, make $ \lambda \downarrow b-a$, then $nm+b\to1$, so
\begin{align*}
&\begin{tikzpicture}[baseline=1.65cm]
\cod{0}{0} \cod{1.5}{a} \cod{3}{2a} \cod{19/8*1.5}{m}
\cod{19/8*1.5+\repw}{nm}
\cod{19/8*1.5+\repw+4/13*\haa}{nm+a}
\cod{19/8*1.5+\repw+8/13*\haa}{nm+2a}
\cod{19/8*1.5+\repw+11/13*\haa}{1+a}
\gbox{1}{0}{1.5}{0} \lper{1}{0}
\gbox{1}{1.5}{3}{1} \gbox{1}{3}{19/8*1.5}{2}
\node at (19/8*1.5-0.1,1.6) [right] {$\times n$};
\gbox{1}{19/8*1.5+\repw}{19/8*1.5+\repw+4/13*\haa}{0}
\gbox{1}{19/8*1.5+\repw+4/13*\haa}{19/8*1.5+\repw+8/13*\haa}{1}
\gbox{1}{19/8*1.5+\repw+8/13*\haa}{19/8*1.5+\repw+11/13*\haa}{2}
\rper{1}{19/8*1.5+\repw+11/13*\haa}
\end{tikzpicture}\\
\to&\begin{tikzpicture}[baseline=1.65cm]
\cod{0}{0} \cod{1.5}{a} \cod{3}{2a} \cod{19/8*1.5}{m}
\cod{19/8*1.5+\repw}{nm}
\cod{19/8*1.5+\repw+1.5}{nm+a\qquad}
\cod{19/8*1.5+\repw+3}{nm+2a\qquad}
\cod{19/8*1.5+\repw+19/8*1.5}{\qquad\qquad (n+1)m}
\gbox{1}{0}{1.5}{0} \lper{1}{0}
\gbox{1}{1.5}{3}{1}
\gbox{1}{3}{19/8*1.5}{2}
\node at (19/8*1.5-.1,1.6) [right] {$\times n$};
\gbox{1}{19/8*1.5+\repw}{19/8*1.5+\repw+1.5}{0}
\gbox{1}{19/8*1.5+\repw+1.5}{19/8*1.5+\repw+3}{1}
\gbox{1}{19/8*1.5+\repw+3}{19/8*1.5+\repw+19/8*1.5}{2}
\rper{1}{19/8*1.5+\repw+19/8*1.5}
\end{tikzpicture}\\
=&\begin{tikzpicture}[baseline=1.65cm]
\cod{0}{0} \cod{1.5}{a} \cod{3}{2a} \cod{19/8*1.5}{m}
\gbox{1}{0}{1.5}{0} \lper{1}{0}
\gbox{1}{1.5}{3}{1}
\gbox{1}{3}{19/8*1.5}{2}
\rper{1}{19/8*1.5}
\end{tikzpicture}.
\end{align*}
Hence the borders $ \lambda =0$ and $ \lambda =b-a$ come from considering the minimum period.

\end{rem}

\begin{rem}\label{rem:overlap}

Following Note~\ref{nt:share}. Still letting $m=a+b$, take $(a,b) =\displaystyle\left(\frac{1}{2n} ,\frac{1}{2n}\right)$ as in $B_n^1$. Then we have
\begin{equation*}
2a=m,\ nm=1,\text{ and }(n-1)m+3a=nm+a=1+a=1+b,
\end{equation*}
so the Nim value function is
\begin{align*}
&\begin{tikzpicture}[baseline=1.65cm]
\cod{0}{0} \cod{1.5}{a} \cod{3}{2a} \cod{19/8*1.5}{m}
\cod{19/8*1.5+\repw}{nm}
\cod{19/8*1.5+\repw+4/13*\haa}{1}
\cod{19/8*1.5+\repw+5/13*\haa}{\begin{tabular}{c}$(n-1)m+3a$\\$\downarrow$\end{tabular}}
\cod{19/8*1.5+\repw+8/13*\haa}{nm+a}
\cod{19/8*1.5+\repw+12/13*\haa}{1+a}
\cod{19/8*1.5+\repw+15/13*\haa}{1+b}
\gbox{1}{0}{1.5}{0} \lper{1}{0}
\gbox{1}{1.5}{3}{1} \gbox{1}{3}{19/8*1.5}{2}
\node at (19/8*1.5-0.1,1.6) [right] {$\times n$};
\gbox{1}{19/8*1.5+\repw}{19/8*1.5+\repw+4/13*\haa}{0}
\gbox{1}{19/8*1.5+\repw+4/13*\haa}{19/8*1.5+\repw+5/13*\haa}{2}
\gbox{1}{19/8*1.5+\repw+5/13*\haa}{19/8*1.5+\repw+8/13*\haa}{3}
\gbox{1}{19/8*1.5+\repw+8/13*\haa}{19/8*1.5+\repw+12/13*\haa}{1}
\gbox{1}{19/8*1.5+\repw+12/13*\haa}{19/8*1.5+\repw+15/13*\haa}{3} \rper{1}{19/8*1.5+\repw+15/13*\haa}
\end{tikzpicture}\\
=&\begin{tikzpicture}[baseline=1.65cm]
\cod{0}{0} \cod{1.5}{a} \cod{3}{2a} 
\cod{3+\repw}{1}
\cod{3+\repw+1.5}{1+a}
\gbox{1}{0}{1.5}{0} \lper{1}{0}
\gbox{1}{1.5}{3}{1}
\node at (3-.1,1.6) [right] {$\times n$};
\gbox{1}{3+\repw}{3+\repw+1.5}{2} \rper{1}{3+\repw+1.5}
\end{tikzpicture}.
\end{align*}
Take $ ( a,b) =\displaystyle\left(\frac{1}{2n} ,\frac{1}{2n}\right)$ as in $R_{n-1}^2$, then we have
\begin{equation*}
2a=m,\ (n-1)m=1-2a,\text{ and }(n-1)m+3a=1+a,
\end{equation*}
so the Nim value function is
\begin{align*}
&\begin{tikzpicture}[baseline=1.65cm]
\cod{0}{0} \cod{1.5}{a} \cod{3}{2a} \cod{19/8*1.5}{m}
\cod{19/8*1.5+2*\repw}{(n-1)m\qquad}
\cod{19/8*1.5+2*\repw+4/14*\haa}{(n-1)m+a\qquad\quad\ }
\cod{19/8*1.5+2*\repw+8/14*\haa}{(n-1)m+2a\qquad\ }
\cod{19/8*1.5+2*\repw+12/14*\haa}{(n-1)m+3a\quad\ }
\cod{19/8*1.5+2*\repw+14/14*\haa}{\qquad\  1+a}
\gbox{1}{0}{1.5}{0} \lper{1}{0}
\gbox{1}{1.5}{3}{1} \gbox{1}{3}{19/8*1.5}{2}
\node at (19/8*1.5-0.1,1.6) [right] {$\times (n-1)$};
\gbox{1}{19/8*1.5+2*\repw}{19/8*1.5+2*\repw+4/14*\haa}{0}
\gbox{1}{19/8*1.5+2*\repw+4/14*\haa}{19/8*1.5+2*\repw+8/14*\haa}{1}
\gbox{1}{19/8*1.5+2*\repw+8/14*\haa}{19/8*1.5+2*\repw+12/14*\haa}{2}
\gbox{1}{19/8*1.5+2*\repw+12/14*\haa}{19/8*1.5+2*\repw+14/14*\haa}{3} \rper{1}{19/8*1.5+2*\repw+14/14*\haa}
\end{tikzpicture}\\
=&\begin{tikzpicture}[baseline=1.65cm]
\cod{0}{0} \cod{1.5}{a} \cod{3}{2a}
\cod{3+2*\repw}{1-2a}
\cod{3+2*\repw+1.5}{1-a}
\cod{3+2*\repw+3}{1}
\cod{3+2*\repw+4.5}{1+a}
\gbox{1}{0}{1.5}{0} \lper{1}{0}
\gbox{1}{1.5}{3}{1}
\node at (3-0.1,1.6) [right] {$\times (n-1)$};
\gbox{1}{3+2*\repw}{3+2*\repw+1.5}{0}
\gbox{1}{3+2*\repw+1.5}{3+2*\repw+3}{1}
\gbox{1}{3+2*\repw+3}{3+2*\repw+4.5}{2}
\rper{1}{3+2*\repw+4.5}
\end{tikzpicture}\\
=&\begin{tikzpicture}[baseline=1.65cm]
\cod{0}{0} \cod{1.5}{a} \cod{3}{2a} 
\cod{3+\repw}{1}
\cod{3+\repw+1.5}{1+a}
\gbox{1}{0}{1.5}{0} \lper{1}{0}
\gbox{1}{1.5}{3}{1}
\node at (3-.1,1.6) [right] {$\times n$};
\gbox{1}{3+\repw}{3+\repw+1.5}{2} \rper{1}{3+\repw+1.5}
\end{tikzpicture},
\end{align*}
the same as above.

\end{rem}

\end{appendices}

\end{document}